\documentclass[12pt]{article}
\usepackage{graphicx}
\usepackage{color,latexsym,amsfonts,amssymb}
\usepackage{bm}     
\usepackage{amsthm} 
\usepackage{amsmath}
\usepackage{booktabs} 
\usepackage{boxedminipage}  
\usepackage{natbib} 
\usepackage{algorithm,algpseudocode}  

\DeclareMathOperator*{\argmax}{argmax}

\newtheorem{theorem}{Theorem}

\newtheorem{lemma}{Lemma}
\newtheorem{definition}{Definition}
\newtheorem{assumption}{Assumption}
\newcommand{\dif}{\mathrm{d}}

\title{Risk-Sensitive Markov Decision Processes with Combined Metrics of Mean and Variance\thanks{This paper was supported in part by the National Natural
Science Foundation of China (61573206, 11931018). This paper was
initially submitted to the journal of Production and Operations
Management (POMS) in July 2019, revised in December 2019, March and
May 2020, and accepted in July 2020.}}
\author{Li Xia\thanks{Li Xia is with the Business School, Sun Yat-Sen University, Guangzhou 510275, P. R. China. (email: xiali5@sysu.edu.cn. \emph{Corresponding Author.})}}
\date{}
\begin{document}
\maketitle

\begin{abstract}
This paper investigates the optimization problem of an infinite
stage discrete time Markov decision process (MDP) with a long-run
average metric considering both mean and variance of rewards
together. Such performance metric is important since the mean
indicates average returns and the variance indicates risk or
fairness. However, the variance metric couples the rewards at all
stages, the traditional dynamic programming is inapplicable as the
principle of time consistency fails. We study this problem from a
new perspective called the sensitivity-based optimization theory. A
performance difference formula is derived and it can quantify the
difference of the mean-variance combined metrics of MDPs under any
two different policies. The difference formula can be utilized to
generate new policies with strictly improved mean-variance
performance. A necessary condition of the optimal policy and the
optimality of deterministic policies are derived. We further develop
an iterative algorithm with a form of policy iteration, which is
proved to converge to local optima both in the mixed and randomized
policy space. Specially, when the mean reward is constant in
policies, the algorithm is guaranteed to converge to the global
optimum. Finally, we apply our approach to study the fluctuation
reduction of wind power in an energy storage system, which
demonstrates the potential applicability of our optimization method.
\end{abstract}

\section{Introduction}\label{section_intro}
The theory of Markov decision processes (MDPs) is widely used to
analyze and optimize the performance of a stochastic dynamic system.
In the framework of MDPs, $R_T$ is denoted as the accumulated
(discounted or undiscounted) rewards until time $T$, which is a
random variable. Many studies of MDPs focus on the discounted or
long-run average performance criteria, i.e., maximization of
$\mathbb E[R_T]$. However, its variance $\sigma^2[R_T]$ is also an
important metric in practical problems, which can reflect the
factors of risk, fairness, quality or safety. For example, the
mean-variance analysis for the risk management in financial
portfolio or hedging \citep{Kouvelis18}, the fairness of customers
served in queueing systems \citep{Avi-Itzhak04}, the power quality
and safety of renewable energy to electricity grid \citep{Li14}, the
optimal advertising of fashion markets and supply chain management
\citep{Chiu18,Zhang20}, etc. Optimization theory of risk-sensitive
MDPs considering not only the average metric but also the variance
metric is an important problem attracting continued attention in the
literature.

However, the optimization of the variance metric does not fit a
standard model of MDPs because of the non-additive variance
function. The value of the variance function will be affected by
actions of not only the current stage but also the future stages
\citep{Bertsekas05,Feinberg02,Puterman94}. Thus, the Bellman
optimality equation does not hold and the principle of time
consistency (roughly speaking, if a decision at time $t+1$ is
optimal, then it should constitute part of optimal decisions at time
$t$) in dynamic programming fails \citep{Ruszczynski10,Shapiro09}.
One main thread of studying the variance optimization problem lays
in the framework of the MDP theory, which aims at minimizing the
variance after the mean performance already achieves optimum. For
the policy subset where the mean already achieves optimum, the
variance minimization problem is equivalent to another standard MDP
with a new average cost criterion. Interested readers can refer to
related papers in the literature, such as the work by \cite{Cao08}
about the $n$-bias optimality of MDPs, the works by
\cite{Guo09b,Guo12} about discounted MDPs, and the work by
\cite{Hernandez99} about average MDPs, just name a few. For these
equivalent MDPs, traditional approaches like dynamic programming or
policy iteration can be used to minimize the variance of MDPs.

Another thread of studying the variance-related optimization problem
is the policy gradient approach for risk-sensitive MDPs
\citep{Prashanth13,Tamar13}, which is more from the algorithmic
viewpoint in reinforcement learning. In this scenario, the policy is
usually parameterized by sophisticated functions, e.g., neural
networks. In general, such optimization problem is not a standard
MDP (a parameterized one) and the policy iteration is not
applicable. For such scenarios, stochastic gradient-based algorithms
are usually adopted to approach to a local optimal solution.
Although neural networks have strong policy representation
capability, such gradient-based algorithms usually converge slowly
and suffer from some intrinsic deficiencies, such as trap of local
optimum, difficulty of selecting step sizes, and sensitivity to
initial points. Seeking policy iteration type algorithms is a
promising attempt to study the MDPs with variance criteria. A recent
study investigates the variance minimization problem regardless of
the optimality of mean performance from the perspective of the
sensitivity-based optimization, where algorithms with a policy
iteration type are developed for the variance minimization problem
of MDPs with averages and parameterized policies, respectively
\citep{Xia16,Xia18b}. It is worth noting that there exist other
risk-sensitive metrics besides the variance-related one. For
example, exponential utility function is also a risk-sensitive
metric since the first two orders of Taylor expansions of
exponential utility include the mean and the variance
\citep{Bauerle2015,Borkar2002,Guo19}. Percentile-related criterion,
VaR (Value at Risk), and CVaR (Conditional VaR) are other widely
studied risk-sensitive metrics in the field of finance and economy
\citep{Chow15,Delage2010,Fu09,Gao17,Hong14}. Motivated by the modern
theory of coherent risk measures \citep{Artzner99}, time-consistent
Markov risk measures are also studied from another point of view in
risk-aware MDPs literature \citep{Ruszczynski10,Ruszczynski06}.
\cite{Haskell13} study the stochastic dominance of accumulated
rewards rather than the expectation of rewards, which is another
interesting way to study the risk-aware metrics in MDPs.

In the literature, some studies consider the performance
optimization of mean and variance together, which is called the
mean-variance optimization. The pioneering work of mean-variance
optimization is initiated by \cite{Markowitz52}, the Nobel Laureate
in Economics. Intensive attention has been paid during past decades
and the mathematical model is extended from single-stage static
problem to multi-stage dynamic one. Markov models are usually
adopted when we study the mean-variance optimization in a stochastic
dynamic scenario. The related works on the dynamic mean-variance
optimization can be categorized into three classes. The first one is
to maximize the mean performance at the condition that the variance
is less than a given amount. The second one is to minimize the
variance while keeping the mean performance larger than a given
amount. The third one is to optimize the combination of the mean and
the variance together, such as maximizing the \emph{Sharpe ratio}
$\frac{\mathbb E[R_T]}{\sigma[R_T]}$ or the weighted combination
$\mathbb E[R_T] - \beta\sigma^2[R_T]$. \cite{Chung94} and
\cite{Sobel94} study the mean-variance optimization in a regime of
long-run average MDPs, by converting the constrained MDP problem to
a mathematical programming problem. Numerical algorithms are also
studied for the multistage mean-variance optimization problem from
the viewpoint of mathematical programming \citep{Parpas07}. For
discounted MDPs, the variance minimization problem with a given
discounted average performance can be converted to an equivalent
discounted MDP with a new cost function
\citep{Huang18,Huo17,Xia18a}, and traditional MDP methods such as
policy iteration are applicable. In some general settings, it has
been shown that the mean-variance optimization problem in MDPs is
NP-hard \citep{Mannor13}. Another way to study the mean-variance
optimization is from the perspective of optimal control, in which
the model is usually continuous time or even continuous state and
the optimality structures of control laws are investigated for
specific scenarios \citep{Zhou00,Zhou04}. In recent years,
reinforcement learning attracts intensive attention for the great
success of AlphaGo. Although most reinforcement learning algorithms
focus on discounted criterion, some policy gradient algorithms are
also investigated to optimize the combined metrics of mean and
variance \citep{Borkar2010,Tamar12}. However, such gradient-based
approach still suffers from the deficiency of local optimum and slow
convergence. On the other hand, as we know, there are three
categories of approaches to solve MDP problems: policy iteration,
value iteration, and policy gradient, where the policy iteration is
a classical and important approach. Although the algorithmic
complexity of the policy iteration is still an open problem
\citep{Littman95}, it is observed that the policy iteration usually
converges very fast in most cases. Therefore, it is promising to
develop policy iteration type algorithms to solve the mean-variance
optimization problem.

In this paper, we study the optimization of the mean-variance
combined metric in the framework of MDPs. The objective is to find
the optimal policy that maximizes the mean minus variance metrics,
$\mathbb{E}[R_T] - \beta \sigma^2[R_T]$. Different coefficients
$\beta$ tradeoff the weights between the mean and the variance.
Since the variance-related metric is not additive, this dynamic
optimization problem does not fit the standard model of MDPs. The
time consistency of dynamic programming is not valid and the Bellman
optimality equation does not hold. We study this problem from a new
perspective called the theory of sensitivity-based optimization
\citep{Cao07}, which is rooted from the theory of perturbation
analysis \citep{Ho91} and largely extended to stochastic dynamic
systems with Markov models. The key idea of the sensitivity-based
optimization is to utilize the performance difference and derivative
information to search the optimal policy. This theory is applicable
for general Markov systems and it does not require an additive cost
function. For the mean-variance combined metric, we derive a
difference formula which quantifies the difference of the combined
metrics under any two policies. The coupling effect caused by the
non-additive variance function can be decoupled by a square term in
the difference formula. Based on the difference formula, we derive a
method to strictly improve the system mean-variance metrics. A
necessary condition of the optimal policy is also derived. We
further prove that the optimal policy with the maximal mean-variance
combined metric can always be achieved in the deterministic policy
space. A policy iteration type algorithm is developed to find the
optimal policy, which is proved to converge to a local optimum both
in the mixed and randomized policy space. Moreover, this algorithm
can find the global optimum if the mean reward remains constant for
different policies. Some exploration techniques are also discussed
to improve the global search capability of the algorithm. Finally,
we use an example in a renewable energy storage system to
demonstrate that our approach can effectively reduce the fluctuation
of the total output power, which improves the output power quality
and reduces the stability risk of power grid. The local convergence
and the global convergence of our approach are both demonstrated in
different scenarios of this example.

The main contributions of the paper are as follows. We use the
sensitivity-based optimization theory to study the performance
optimization of mean-variance combined metrics in Markov systems,
which is a new perspective different from the traditional dynamic
programming. Although the principle of dynamic programming is not
applicable, we derive a difference formula to quantify the changing
behaviors of mean-variance combined metrics with respect to
different policies. Optimality conditions and a policy iteration
type algorithm for optimal policies are also derived. To the best of
our knowledge, our paper is the first work that a policy iteration
type algorithm is developed to optimize the mean-variance combined
metrics in MDPs. Compared with traditional gradient-based
algorithms, policy iteration usually has a much faster convergence
speed, which is partly demonstrated by our numerical experiments.

This paper is substantially extended based on its conference version
\citep{Xia19}, where the main theorems, algorithms, and experiments
are not presented. The difference of this paper from our previous
work \citep{Xia16} is that this paper optimizes both the mean and
variance metrics together, while \cite{Xia16} only minimizes the
variance metric without any consideration of the mean performance.
It is important to consider the mean and variance metrics together
in real life. Many works about dynamic mean-variance optimization
focus on the variance of accumulated discounted rewards at a
terminal stage, where the contraction operator is a key tool. There
are much less works studying the mean-variance optimization of
rewards in the sense of long-run averages, where the contraction
operator is not applicable. Therefore, our work on the long-run
mean-variance combined metrics could give a complementary result to
the whole theory of mean-variance optimization. We hope that this
paper could shed light on studying the dynamic mean-variance
optimization, especially in the scenario with long-run average
metrics (instead of discounted ones) and policy iteration type
algorithms (instead of gradient ones).

The remainder of the paper is organized as follows. In
Section~\ref{section_model}, we formulate the performance
optimization problem of mean-variance combined metrics under the
framework of MDPs. The main challenge of this problem is also
discussed. In Section~\ref{section_result}, we use the
sensitivity-based optimization theory to derive the difference
formula for the mean-variance combined metrics. Optimality
structures and necessary condition of the optimal policy are also
derived. Furthermore, we develop an iterative algorithm to find the
optimal policy in Section~\ref{section_algo}. The convergence
analysis of the algorithm is also studied. In
Section~\ref{section_expriment}, we use an example of renewable
energy storage systems to numerically illustrate the local and
global convergence of our approach, respectively. Finally, we
conclude this paper and discuss some future research topics in
Section~\ref{section_conclusion}.

\section{Problem Formulation}\label{section_model}
We consider a discrete time Markov decision process. The state space
and the action space are both finite and denoted as $\mathcal
S:=\{1,2, \cdots, S\}$ and $\mathcal A:=\{a_1,a_2,\cdots,a_A\}$,
respectively. In general, $\mathcal A$ can also be state-dependent,
i.e., $\mathcal A(i)$ and $i \in \mathcal S$, but we omit this case
for simplicity of representation. A stationary policy is a mapping
from the state space to the action space, $d: \mathcal S \rightarrow
\mathcal A$, where $d$ is deterministic since its optimality is
proved in the later section. The stationary policy space is denoted
as $\mathcal D$. At each state $i \in \mathcal S$, an action $a =
d(i)$ will be adopted if policy $d \in \mathcal D$ is being
employed. Then the system will transit to the next state $j$ with
transition probability $p^a(i,j)$ and an instant reward $r(i,a)$
will be incurred, $i,j\in \mathcal S$. We denote $\bm P^d$ as the
transition probability matrix whose elements are $p^{d(i)}(i,j)$'s,
$i,j\in \mathcal S$. We know that $\bm P^d$ is an $S$-by-$S$ matrix
and $\bm P^d \bm 1 = \bm 1$, where $\bm 1$ is an $S$-dimensional
column vector with all elements being 1. We further denote $\bm r^d$
as an $S$-dimensional column vector whose elements are
$r(i,d(i))$'s, $i \in \mathcal S$. Note that we may also remove the
superscript ``$^d$" and use $\bm P$ and $\bm r$ for notation
simplicity. We make the following assumption about the MDP model in
this paper.

\begin{assumption}\label{assumption1}
We only focus on stationary policies  in this paper and the MDP
under any policy $d \in \mathcal D$ is always ergodic.
\end{assumption}

We denote $\pi^d(i)$ as the stationary probability that the system
stays at state $i$ when policy $d$ is employed. Therefore, the
vector of steady state distribution is denoted as an $S$-dimensional
row vector $\bm \pi^d:=(\pi^d(1),\pi^d(2),\cdots,\pi^d(S))$, which
can be determined by the following equations
\begin{equation}\label{eq_pi}
\begin{array}{l}
\bm \pi^d \bm P^d = \bm \pi^d, \\
\bm \pi^d \bm 1 = 1.
\end{array}
\end{equation}
We define the \emph{long-run average} performance of the MDP
\begin{equation}\label{eq_J}
J_{\mu}^d := \bm \pi^d \bm r^d = \lim\limits_{T \rightarrow \infty}
\frac{1}{T} \mathbb{E}^d\left\{ \sum_{t=0}^{T-1} r(X_t,A_t)
\right\},
\end{equation}
where $\mathbb{E}^d$ indicates the expectation under policy $d$,
$X_t$ is the system state at time $t$, $A_t = d(X_t)$ and it is the
action at time $t$. With Assumption~\ref{assumption1}, the MDP under
any policy $d$ is always ergodic, thus $J_{\mu}^d$ is independent of
the initial state $X_0$.

We also define the \emph{steady state variance} (or called
\emph{long-run variance}) of rewards of the MDP
\citep{Chung94,Sobel94}
\begin{equation}\label{eq_var}
J_{\sigma}^d := \lim\limits_{T \rightarrow \infty} \frac{1}{T}
\mathbb{E}^d\left\{ \sum_{t=0}^{T-1} \left[r(X_t,A_t) -
J_{\mu}^d\right]^2 \right\}.
\end{equation}
Note that the steady state variance $J_{\sigma}^d$ defined above is
different from the variance of the accumulated rewards
$R_T:=\sum_{t=0}^{T}\alpha^t r(X_t,A_t)$ with discount factor
$\alpha$. The latter definition is also widely studied in the MDP
literature and it reflects the variation of the terminal wealth at
time $T$. However, it does not reflect the variation during the
reward process.

In this paper, we study the following \emph{mean-variance combined
metric}
\begin{equation}\label{eq_Jmuvar}
J_{\mu,\sigma}^d := J^d_{\mu} - \beta J_{\sigma}^d,
\end{equation}
where $\beta > 0$ and it is a weight between $J^d_{\mu}$ and
$J_{\sigma}^d$. With definitions (\ref{eq_J}) and (\ref{eq_var}), we
can rewrite (\ref{eq_Jmuvar}) as
\begin{equation}
J_{\mu,\sigma}^d = \lim\limits_{T \rightarrow \infty} \frac{1}{T}
\mathbb{E}^d\left\{ \sum_{t=0}^{T-1} r(X_t,A_t) - \beta
[r(X_t,A_t)-J_{\mu}^d]^2 \right\} = \bm \pi^d \bm f_{\mu,\sigma}^d,
\end{equation}
where $\bm f_{\mu,\sigma}^d$ is a column vector with $S$-dimension
and its element $f_{\mu,\sigma}^d(i)$ is defined as
\begin{equation}\label{eq_f}
f_{\mu,\sigma}^d(i) := r(i,d(i)) - \beta [r(i,d(i))-J_{\mu}^d]^2,
\quad i \in \mathcal S.
\end{equation}
Thus, the optimization of the mean-variance combined metric is
represented by the long-run average performance of the MDP with the
new cost function (\ref{eq_f}). The objective is to find the optimal
policy $d^*$ which maximizes the associated value of
$J_{\mu,\sigma}^{d^*}$, i.e.,
\begin{equation}\label{eq_prob}
d^* = \argmax\limits_{d \in \mathcal D} \{ J_{\mu,\sigma}^{d} \}.
\end{equation}

The problem (\ref{eq_prob}) can also be viewed as a multi-objective
optimization problem, one objective is the mean, and the other one
is the variance. The value of $(J_{\mu}^{d^*}, J_{\sigma}^{d^*})$
associated with the optimal policy $d^*$ derived in (\ref{eq_prob})
is a \emph{Pareto} optimum of the multi-objective optimization
problem. Different $\beta$ will induce different solution to
(\ref{eq_prob}). All the associated pairs $(J_{\mu}^{d^*},
J_{\sigma}^{d^*})$ compose the \emph{Pareto frontier} (or efficient
frontier) of the mean-variance optimization problem.

The optimization problem (\ref{eq_prob}) is of significance in
practical applications, as the variance-related metric has rich
physical meanings. In financial systems, the variance metric can be
used to quantify the \emph{risk} of assets, which is widely used in
portfolio management and financial hedging \citep{Kouvelis18}. In
service systems, the variance of customers' waiting time can be used
to quantify the \emph{fairness} of service. The study of service
fairness is an interesting topic in queueing theory
\citep{Avi-Itzhak04}. In industrial engineering, the variance can be
used to quantify the \emph{quality} of products, such as the
variation minimization control in chemical process industry
\citep{Harrison09}. In power grid, the variance can be used to
quantify the \emph{stability} of output power, which is critical for
the safety of the grid \citep{Li14}. We also use an example of power
grid to demonstrate the applicability of our optimization approach
in Section~\ref{section_expriment}.

However, it is challenging to solve (\ref{eq_prob}). The
optimization problem (\ref{eq_prob}) does not fit the standard model
of MDPs, because its cost function (\ref{eq_f}) includes a term of
$J_{\mu}^d$. The value of $J_{\mu}^d$ is affected by not only the
action selection at the current stage, but also those at future
stages. Thus, the action selection in future stages will affect the
``instant" cost (\ref{eq_f}) at the current stage, which indicates
that (\ref{eq_f}) is not \emph{additive} and violates the
requirement of a standard MDP model. The Bellman optimality equation
does not hold and the principle of dynamic programming fails for
this problem \citep{Puterman94}. In the next section, we will study
this optimization problem with mean-variance combined metrics from a
new perspective, called the sensitivity-based optimization theory.

\section{Sensitivity-Based Optimization}\label{section_result}
The theory of sensitivity-based optimization is proposed by
\cite{Cao07} and it provides a new perspective to study the
performance optimization of Markov systems. The origin of this
theory can be traced back to the theory of \emph{perturbation
analysis} \citep{Ho91} and its basic idea is to exploit sensitivity
information from system sample paths, thus to guide the system
optimization. The sensitivity information includes not only
performance gradients, but also performance difference. A key result
of the sensitivity-based optimization is the performance difference
formula which quantifies the difference of the system performance
under any two policies or parameter settings. This theory does not
require a standard model of MDPs and it is valid for Markov systems
with general controls, even works in some scenarios where the
traditional dynamic programming fails \citep{Cao07,Xia14}. One of
the main contributions of our paper is to derive a mean-variance
combined performance difference formula for solving (\ref{eq_prob}),
which has an elegant form with a square term to handle the failure
of dynamic programming.

\subsection{Performance Difference of Mean-Variance Combined Metrics}
The theory of sensitivity-based optimization has a fundamental
quantity called \emph{performance potential} $g^d_{\mu,\sigma}(i)$,
which is defined as below \citep{Cao07}
\begin{equation}\label{eq_g}
g^d_{\mu,\sigma}(i) := \lim\limits_{T \rightarrow \infty} \mathbb
E^d \left\{ \sum_{t=0}^{T} [f^d_{\mu,\sigma}(X_t) -
J^d_{\mu,\sigma}]\Big|{X_0=i} \right\}, \quad  i \in \mathcal S, d
\in \mathcal D.
\end{equation}
We further denote $\bm g^d_{\mu,\sigma}$ as an $S$-dimensional
column vector whose elements are $g^d_{\mu,\sigma}(i)$'s, $i \in
\mathcal S$. From (\ref{eq_g}), $g^d_{\mu,\sigma}(i)$ can be
understood as a quantity which accumulates the advantages of the
system rewards over the average reward $J^d_{\mu,\sigma}$, caused by
the condition of the initial state $i$. If the value of
$g^d_{\mu,\sigma}(i)$ is positively large, it indicates that the
initial state $i$ has a good position with a high potential,
compared with the long-run average level $J^d_{\mu,\sigma}$.
Therefore, $\bm g^d_{\mu,\sigma}$ is also called \emph{bias} or
\emph{relative value function} in the classical MDP theory
\citep{Puterman94}. To compute the value of $g^d_{\mu,\sigma}(i)$ in
(\ref{eq_g}), we have to obtain the long-run average
$J^d_{\mu,\sigma}$ first. Meanwhile, the value of $J^d_{\mu,\sigma}$
depends on the calculation of $J^d_{\mu}$ and $J^d_{\sigma}$
according to (\ref{eq_Jmuvar}). Substituting (\ref{eq_Jmuvar}) and
(\ref{eq_f}) into (\ref{eq_g}), we can further derive
\begin{eqnarray}\label{eq9}
g^d_{\mu,\sigma}(i) &=& \lim\limits_{T \rightarrow \infty} \mathbb
E^d \left\{ \sum_{t=0}^{T} [r(X_t) - \beta (r(X_t)-J^d_{\mu})^2 -
J^d_{\mu} + \beta J^d_{\sigma}]\Big|{X_0=i} \right\} \nonumber\\
&=& \lim\limits_{T \rightarrow \infty} \mathbb E^d \left\{
\sum_{t=0}^{T} [r(X_t) - J^d_{\mu}]\Big|{X_0=i} \right\} - \beta
\lim\limits_{T \rightarrow \infty} \mathbb E^d \left\{
\sum_{t=0}^{T}
[(r(X_t)-J^d_{\mu})^2 - J^d_{\sigma}]\Big|{X_0=i} \right\} \nonumber\\
&=& g^d_{\mu}(i) - \beta g^d_{\sigma}(i),
\end{eqnarray}
where $g^d_{\mu}(i)$ and $g^d_{\sigma}(i)$ are the corresponding
performance potentials of the MDP with cost functions $r(i)$ and
$(r(i) - J^d_{\mu})^2$, respectively, using similar definitions in
(\ref{eq_g}) as follows
\begin{equation}
g^d_{\mu}(i) := \lim\limits_{T \rightarrow \infty} \mathbb E^d
\left\{ \sum_{t=0}^{T} [r(X_t) - J^d_{\mu}]\Big|{X_0=i} \right\},
\quad  i \in \mathcal S, d \in \mathcal D, \nonumber
\end{equation}
\begin{equation}
g^d_{\sigma}(i) := \lim\limits_{T \rightarrow \infty} \mathbb E^d
\left\{ \sum_{t=0}^{T} [(r(X_t)-J^d_{\mu})^2 -
J^d_{\sigma}]\Big|{X_0=i} \right\}, \quad  i \in \mathcal S, d \in
\mathcal D. \nonumber
\end{equation}
Equation (\ref{eq9}) indicates that the performance potentials
conserve the linear additivity if the cost function can be linearly
separated into different parts.

Using the strong Markov property and extending the summation terms
of (\ref{eq_g}) at time $t=0$, we derive a recursive equation
\begin{eqnarray}\label{eq_g2}
g^d_{\mu,\sigma}(i) &=& f^d_{\mu,\sigma}(i) - J^d_{\mu,\sigma} +
\lim\limits_{T \rightarrow \infty} \mathbb E^d \left\{
\sum_{t=1}^{T} [f^d_{\mu,\sigma}(X_t) -
J^d_{\mu,\sigma}]\Big|{X_0=i} \right\}
\nonumber\\
&=& f^d_{\mu,\sigma}(i) - J^d_{\mu,\sigma} + \sum_{j \in \mathcal S}
p^{d(i)}(i,j) g^d_{\mu,\sigma}(j),
\end{eqnarray}
where the second equality is derived by recursively applying
(\ref{eq_g}). We can further rewrite (\ref{eq_g2}) in a matrix form
and derive the following \emph{Poisson equation}:
\begin{equation}\label{eq_poisson}
\bm g^d_{\mu,\sigma} = \bm f^d_{\mu,\sigma} - J^d_{\mu,\sigma} \bm
1+ \bm P^d \bm g^d_{\mu,\sigma}.
\end{equation}
It is known that $\bm P^d$ is a stochastic matrix and its rank is
$S-1$. Thus, we can set $g^d_{\mu,\sigma}(1)=c$ and
(\ref{eq_poisson}) can be numerically solved with a unique solution,
where $c$ is any fixed constant. That is, $\bm g^d_{\mu,\sigma} + c
\bm 1$ is also a solution to (\ref{eq_poisson}) for any constant
$c$. Moreover, we can also use the definition (\ref{eq_g}) to online
learn or estimate the value of $\bm g^d_{\mu,\sigma}$ based on
system sample paths (see details in Chapter 3 of \citep{Cao07} and
reinforcement learning \citep{Sutton18}).

Next, we are ready to study the mean-variance combined performance
difference between $J^d_{\mu,\sigma}$ and $J^{d'}_{\mu,\sigma}$ when
the policy of MDPs is changed from $d$ to $d'$. To simplify
notations, we omit the superscript ``$^{d}$" by default and use the
superscript ``$^{'}$" instead of ``$^{d'}$" in the rest of the paper
if applicable. We write the cost function (\ref{eq_f}) in a vector
form
\begin{equation}\label{eq_f2}
\bm f_{\mu,\sigma} = \bm r - \beta (\bm r - J_{\mu} \bm
1)^2_{\odot},
\end{equation}
where $(\bm r - J_{\mu} \bm 1)^2_{\odot}$ is the component-wise
square of vector $(\bm r - J_{\mu} \bm 1)$, i.e.,
\begin{equation}
(\bm r - J_{\mu} \bm 1)^2_{\odot} :=
[(r(1)-J_{\mu})^2,(r(2)-J_{\mu})^2,\cdots,(r(S)-J_{\mu})^2]^T.
\nonumber
\end{equation}
Thus, the mean-variance combined performance metric under policy $d$
is
\begin{equation}
J_{\mu,\sigma} = \bm \pi \bm f_{\mu,\sigma} = \bm \pi [\bm r - \beta
(\bm r - J_{\mu} \bm 1)^2_{\odot}].
\end{equation}
Similarly, the mean-variance combined performance metric under
policy $d'$ is
\begin{equation}\label{eq14}
J'_{\mu,\sigma} = \bm \pi' \bm f'_{\mu,\sigma} = \bm \pi' [\bm r' -
\beta (\bm r' - J'_{\mu} \bm 1)^2_{\odot}].
\end{equation}
Substituting (\ref{eq_f2}) into (\ref{eq_poisson}), we can rewrite
the Poisson equation as
\begin{equation}\label{eq_poisson2}
\bm g_{\mu,\sigma} = \bm r - \beta (\bm r - J_{\mu} \bm 1)^2_{\odot}
- J_{\mu,\sigma} \bm 1+ \bm P \bm g_{\mu,\sigma}.
\end{equation}
By left-multiplying $\bm \pi'$ on both sides of (\ref{eq_poisson2}),
we have
\begin{equation}\label{eq16}
\bm \pi' \bm P' \bm g_{\mu,\sigma} = \bm \pi' \bm r - \beta \bm \pi'
(\bm r - J_{\mu} \bm 1)^2_{\odot} - J_{\mu,\sigma} + \bm \pi' \bm P
\bm g_{\mu,\sigma},
\end{equation}
where we use the fact $\bm \pi' \bm P' = \bm \pi'$ and $\bm \pi' \bm
1 = 1$. By letting (\ref{eq14}) minus (\ref{eq16}), we obtain
\begin{equation}\label{eq_dif}
J'_{\mu,\sigma} - J_{\mu,\sigma} = \bm \pi'\Big[ (\bm P' - \bm P)\bm
g_{\mu,\sigma} + \bm r' - \beta (\bm r' - J'_{\mu} \bm 1)^2_{\odot}
- \bm r + \beta (\bm r - J_{\mu} \bm 1)^2_{\odot} \Big].
\end{equation}

Furthermore, we extend $\bm \pi' (\bm r' - J'_{\mu} \bm
1)^2_{\odot}$ and derive
\begin{eqnarray}\label{eq18}
\bm \pi' (\bm r' - J'_{\mu} \bm 1)^2_{\odot} &=& \sum_{i \in
\mathcal S} \Big[ \pi'(i) r'^2(i) - 2 \pi'(i) r'(i) J'_{\mu} +
\pi'(i) J'^2_{\mu} \Big] \nonumber\\
&=& \sum_{i \in \mathcal S} \pi'(i) r'^2(i) - 2 J'_{\mu} \sum_{i \in
\mathcal S} \pi'(i) r'(i) + J'^2_{\mu} \sum_{i \in \mathcal S}
\pi'(i) \nonumber\\
&=& \sum_{i \in \mathcal S} \pi'(i) r'^2(i) - J'^2_{\mu} \nonumber\\
&=& \sum_{i \in \mathcal S} \pi'(i) r'^2(i) - 2 J_{\mu} \sum_{i \in
\mathcal S} \pi'(i) r'(i) + J^2_{\mu} \sum_{i \in \mathcal S}
\pi'(i) - (J'_{\mu}-J_{\mu})^2 \nonumber\\
&=& \bm \pi' (\bm r' - J_{\mu} \bm 1)^2_{\odot} -
(J'_{\mu}-J_{\mu})^2,
\end{eqnarray}
where the fact $\sum_{i \in \mathcal S} \pi'(i) r'(i) = J'_{\mu}$ is
utilized. By substituting (\ref{eq18}) into (\ref{eq_dif}), we
directly derive the following lemma about \emph{the mean-variance
combined performance difference formula}.

\begin{lemma}\label{lemma1}
If an MDP's policy is changed from $d$ to $d'$, the associated $(\bm
P, \bm r)$ is changed to $(\bm P', \bm r')$, then the difference of
the system mean-variance combined metrics is quantified by
\begin{eqnarray}\label{eq_dif2}
J'_{\mu,\sigma} - J_{\mu,\sigma} = \bm \pi'\Big[ (\bm P' - \bm P)\bm
g_{\mu,\sigma} + \bm r' - \beta (\bm r' - J_{\mu} \bm 1)^2_{\odot} -
\bm r + \beta (\bm r - J_{\mu} \bm 1)^2_{\odot} \Big] + \beta
(J'_{\mu}-J_{\mu})^2.
\end{eqnarray}
\end{lemma}

An important feature of (\ref{eq_dif2}) is that every element in the
square brackets is given or computable under the current policy $d$,
although the explicit values of every possible $\bm \pi'$ and
$J'_{\mu}$ are unknown. The mean-variance combined performance
difference formula (\ref{eq_dif2}) is a key result of our paper to
solve (\ref{eq_prob}). It clearly quantifies the relationship
between the long-run mean-variance combined performance and the
adopted policies, where different policies are represented by
different $(\bm P, \bm r)$'s. It is worth noting that the square
term $\beta (J'_{\mu}-J_{\mu})^2$ is critical for the optimization
analysis in the rest of the paper because of the fact $\beta
(J'_{\mu}-J_{\mu})^2 \geq 0$. For the current system with policy
$(\bm P, \bm r)$, we can compute or learn the associated values of
$J_{\mu}$ and $\bm g_{\mu,\sigma}$. Although the value of $\pi'(i)$
is unknown, we always have the fact $\pi'(i)>0$ for ergodic MDPs
with Assumption~\ref{assumption1}. If we choose a new policy $d'$
with proper $(\bm P', \bm r')$ such that all the elements of the
column vector represented by the square brackets of (\ref{eq_dif2})
are positive, then we have $\bm \pi'\Big[ (\bm P' - \bm P)\bm
g_{\mu,\sigma} + \bm r' - \beta (\bm r' - J_{\mu} \bm 1)^2_{\odot} -
\bm r + \beta (\bm r - J_{\mu} \bm 1)^2_{\odot} \Big]
> 0$, which indicates $J'_{\mu,\sigma} - J_{\mu,\sigma} > 0$ and the
system performance under this new policy is improved. Repeating such
operations, we can iteratively improve the policy, which also
motivates the development of an optimization algorithm in the next
section.

\noindent\textbf{Remark~1.} To investigate the system behavior under
a new policy $(\bm P', \bm r')$, we usually need to obtain $\bm
\pi'$, $J'_{\mu}$, and $J'_{\mu,\sigma}$ with high computation
burdens. Fortunately, (\ref{eq_dif2}) avoids such burden by
utilizing the fact that $\bm \pi'$ and $(J'_{\mu}-J_{\mu})^2$ are
always nonnegative (although we do not know their explicit values).
This also reveals sensitivity information to guide the optimization
algorithm, which we call difference sensitivity (compared with
derivative sensitivity).

With the performance difference formula (\ref{eq_dif2}), we can
derive the following theorem about generating improved policies.
\begin{theorem}\label{theorem1}
If a new policy $d'$ with $(\bm P', \bm r')$ satisfies
\begin{equation}\label{eq_newp}
\sum_{j \in \mathcal S} p'(i,j) g_{\mu,\sigma}(j) + r'(i) - \beta
(r'(i) - J_{\mu})^2 \geq \sum_{j \in \mathcal S} p(i,j)
g_{\mu,\sigma}(j) + r(i) - \beta (r(i) - J_{\mu})^2, \quad \forall i
\in \mathcal S,
\end{equation}
then we have $J'_{\mu,\sigma} \geq J_{\mu,\sigma}$. If the
inequality strictly holds for at least one state $i$, then we have
$J'_{\mu,\sigma} > J_{\mu,\sigma}$.
\end{theorem}
%
%

The proof of Theorem~\ref{theorem1} is straightforward based on
(\ref{eq_dif2}). We omit the detailed proof for simplicity.
Theorem~\ref{theorem1} indicates an approach by which we can
generate an improved policy $d'$ based on the condition
(\ref{eq_newp}). That is, after we obtain the value of $\bm
g_{\mu,\sigma}$ and $J_{\mu}$ under the current policy $d$, we can
find a new policy with proper $(\bm P', \bm r')$ such that the value
of $\sum_{j \in \mathcal S} p'(i,j) g_{\mu,\sigma}(j) + r'(i) -
\beta (r'(i) - J_{\mu})^2$ at each state $i$ is as large as
possible. The mean-variance combined performance of the system under
this new policy will be improved, as guaranteed by
Theorem~\ref{theorem1}.

With Theorem~\ref{theorem1} and difference formula (\ref{eq_dif2}),
we can further derive the following theorem about a \emph{necessary
condition} of the optimal policy for (\ref{eq_prob}).
\begin{theorem}\label{theorem2}
For the optimization problem (\ref{eq_prob}), the optimal policy $d$
with $(\bm P, \bm r)$ must satisfy
\begin{equation}\label{eq_nec}
\sum_{j \in \mathcal S} p'(i,j) g_{\mu,\sigma}(j) + r'(i) - \beta
(r'(i) - J_{\mu})^2 \leq \sum_{j \in \mathcal S} p(i,j)
g_{\mu,\sigma}(j) + r(i) - \beta (r(i) - J_{\mu})^2, \quad \forall i
\in \mathcal S
\end{equation}
for any policy $d' \in \mathcal D$.
\end{theorem}

\begin{proof}
This theorem can be proved by using contradiction. Assume that the
condition (\ref{eq_nec}) does not hold for the optimal policy $d$
with $(\bm P, \bm r)$. That is, for some state, say state $k$, there
exists an action $a'$ such that
\begin{equation} \label{eq24}
\sum_{j \in \mathcal S}p'(k,j)g_{\mu,\sigma}(j) + r'(k) - \beta
(r'(k) - J_{\mu})^2 > \sum_{j \in \mathcal S}p(k,j)g_{\mu,\sigma}(j)
+ r(k) - \beta (r(k) - J_{\mu})^2.
\end{equation}
Therefore, we can construct a new policy $d'$ as follows: for state
$k$, choose action $a'$; for other states, choose exactly the same
action as that of the optimal policy $d$. Therefore, substituting
(\ref{eq24}) into (\ref{eq_dif2}), we obtain
\begin{eqnarray}
J'_{\mu,\sigma} - J_{\mu,\sigma} \hspace{-0.2cm}&=&\hspace{-0.2cm}
\pi'(k)\left[ \sum_{j \in \mathcal S}(p'(k,j) -
p(k,j))g_{\mu,\sigma}(j) + r'(k) - \beta (r'(k) - J_{\mu})^2 - r(k)
+ \beta (r(k) -
J_{\mu})^2 \right] \nonumber\\
&& + \beta (J'_{\mu}-J_{\mu})^2 \nonumber\\
&>& 0 + \beta (J'_{\mu}-J_{\mu})^2 \geq  0. \nonumber
\end{eqnarray}
Therefore, we have $J'_{\mu,\sigma} > J_{\mu,\sigma}$, which means
that the policy $d'$ is better than the optimal policy $d$. This
contradicts the assumption that $d$ is the optimal policy. Thus, the
assumption does not hold and the theorem is proved.
\end{proof}

Note that the condition (\ref{eq_nec}) is only a necessary condition
for the optimal policy, not a sufficient condition. The reason comes
from the square term $\beta(J'_{\mu}-J_{\mu})^2$ in the difference
formula (\ref{eq_dif2}). It is possible to construct a new policy
$d'$ such that the first part of the right-hand-side of
(\ref{eq_dif2}) is negative, while the second part
$\beta(J'_{\mu}-J_{\mu})^2$ is relatively larger. Thus, $d'$ is
better than $d$, while the condition (\ref{eq_nec}) is unsatisfied.

\noindent\textbf{Remark~2.} If the system mean rewards under any
policy in $\mathcal D$ are the same, (\ref{eq_nec}) becomes the
\emph{necessary and sufficient condition} of the optimal policy.
This is also consistent with the existing results in the literature
in which the optimization objective is to minimize the variance
after the mean reward is already optimal \citep{Cao08,Guo09b}.

\subsection{Performance Derivative of Mean-Variance Combined Metrics}
In the theory of sensitivity-based optimization, performance
derivative formula is another fundamental concept compared with the
difference formula. Based on the performance derivative,
gradient-based methods can be utilized to solve the mean and
variance optimization problems in the literature since the
traditional dynamic programming method cannot be directly applied
\citep{Prashanth13,Tamar12}. Below, we study the performance
derivative formula of this mean-variance combined metric with
respect to randomized parameters.

First, we define a special case of randomized policy by using the
concept of \emph{mixed policy} in MDPs and game theory
\citep{Feinberg02}. Consider any two deterministic policies $d$ and
$d'$ with $(\bm P, \bm r)$ and $(\bm P', \bm r')$, respectively. We
define a mixed policy $d^{\delta,d'}$ which adopts policy $d$ with
probability $1-\delta$ and adopts policy $d'$ with probability
$\delta$, $0 \leq \delta \leq 1$. Obviously, we have $d^{0,d'} = d$
and $d^{1,d'} = d'$. For simplicity, we denote the transition
probability matrix and the reward function under this mixed policy
as $\bm P^{\delta}$ and $\bm r^{\delta}$, respectively. We can
verify that
\begin{equation}\label{eq25}
\begin{array}{l}
\bm P^{\delta} = \bm P + \delta (\bm P' - \bm P), \\
\bm r^{\delta} = \bm r + \delta (\bm r' - \bm r). \\
\end{array}
\end{equation}
The steady state distribution and the mean reward of the system
under this mixed policy are denoted as $\bm \pi^{\delta}$ and
$J^{\delta}_{\mu}$, respectively. We have
\begin{equation}\label{eq26}
\begin{array}{l}
\bm \pi^{\delta} \bm P^{\delta} = \bm \pi^{\delta}, \\
J^{\delta}_{\mu} = \bm \pi^{\delta} \bm r^{\delta} = \bm
\pi^{\delta} [\bm r + \delta (\bm r' - \bm r)].
\end{array}
\end{equation}
With this mixed policy $d^{\delta,d'}$, the cost function of the
mean-variance combined metric can be written as
\begin{equation}
f^{\delta}(i) = (1-\delta)[r(i) - \beta (r(i) - J^{\delta}_{\mu})^2]
+ \delta[r'(i) - \beta (r'(i) - J^{\delta}_{\mu})^2], \quad \forall
i \in \mathcal S. \nonumber
\end{equation}
The vector of the cost function is denoted as $\bm f^{\delta}$. The
long-run average performance of the mean-variance combined metric
under this mixed policy $d^{\delta,d'}$ can be written as
\begin{equation}
J^{\delta}_{\mu,\sigma} = \sum_{i \in \mathcal S}
\pi^{\delta}(i)f^{\delta}(i) = \bm \pi^{\delta} \bm f^{\delta}.
\nonumber
\end{equation}

Similar to the derivation procedure of (\ref{eq_dif}),
left-multiplying $\bm \pi^{\delta}$ on both sides of
(\ref{eq_poisson2}), we derive the performance difference formula
between policies $d^{\delta,d'}$ and $d$ as follows
\begin{equation}\label{eq_30}
J^{\delta}_{\mu,\sigma} - J_{\mu,\sigma} = \bm \pi^{\delta}\Big\{
(\bm P^{\delta} - \bm P) \bm g_{\mu,\sigma} + (1-\delta)[\bm r -
\beta(\bm r - J^{\delta}_{\mu}\bm 1)^2_{\odot}] + \delta [\bm r' -
\beta(\bm r' - J^{\delta}_{\mu} \bm 1)^2_{\odot}] - \bm r +
\beta(\bm r - J_{\mu} \bm 1)^2_{\odot} \Big\}
\end{equation}
Substituting (\ref{eq26}) into (\ref{eq_30}), we derive
\begin{eqnarray}\label{eq_31}
J^{\delta}_{\mu,\sigma} - J_{\mu,\sigma} &=& \bm \pi^{\delta}\Big\{
(\bm P^{\delta} - \bm P) \bm g_{\mu,\sigma} + \delta [\bm r' -
\beta(\bm r' - J^{\delta}_{\mu} \bm 1)^2_{\odot}] -\delta[\bm r -
\beta(\bm r - J^{\delta}_{\mu} \bm 1)^2_{\odot}] \Big\} \nonumber\\
&& + \bm \pi^{\delta} \beta[ (\bm r - J_{\mu} \bm 1)^2_{\odot} -
(\bm r - J^{\delta}_{\mu} \bm 1)^2_{\odot} ] \nonumber\\
&=& \bm \pi^{\delta}\Big\{ (\bm P^{\delta} - \bm P) \bm
g_{\mu,\sigma} + \delta [\bm r' - \beta(\bm r' - J_{\mu} \bm
1)^2_{\odot}] -\delta[\bm r -
\beta(\bm r - J_{\mu} \bm 1)^2_{\odot}] \Big\} \nonumber\\
&& + \bm \pi^{\delta} \beta \delta [2\bm r' J^{\delta}_{\mu}-2\bm r
J^{\delta}_{\mu} - 2\bm r'J_{\mu} + 2\bm r J_{\mu}] + \bm
\pi^{\delta} \beta[ (\bm r - J_{\mu} \bm 1)^2_{\odot} -
(\bm r - J^{\delta}_{\mu} \bm 1)^2_{\odot} ] \nonumber\\
&=& \bm \pi^{\delta}\Big\{ (\bm P^{\delta} - \bm P) \bm
g_{\mu,\sigma} + \delta [\bm r' - \beta(\bm r' - J_{\mu} \bm
1)^2_{\odot}] -\delta[\bm r -
\beta(\bm r - J_{\mu} \bm 1)^2_{\odot}] \Big\} \nonumber\\
&& + 2\bm \pi^{\delta} \beta  [\bm r + \delta (\bm r' - \bm r)]
J^{\delta}_{\mu} - 2\bm \pi^{\delta} \beta  [\bm r + \delta (\bm r'
- \bm r)] J_{\mu} + \beta [J^2_{\mu} - (J^{\delta}_{\mu})^2] \nonumber\\
&=& \bm \pi^{\delta}\Big\{ (\bm P^{\delta} - \bm P) \bm
g_{\mu,\sigma} + \delta [\bm r' - \beta(\bm r' - J_{\mu} \bm
1)^2_{\odot}] -\delta[\bm r -
\beta(\bm r - J_{\mu} \bm 1)^2_{\odot}] \Big\} \nonumber\\
&& + \beta [J^2_{\mu} - 2J_{\mu}J^{\delta}_{\mu} + (J^{\delta}_{\mu})^2] \nonumber\\
&=& \bm \pi^{\delta}\Big\{ (\bm P^{\delta} - \bm P) \bm
g_{\mu,\sigma} + \delta [\bm r' - \beta(\bm r' - J_{\mu} \bm
1)^2_{\odot}] -\delta[\bm r -
\beta(\bm r - J_{\mu} \bm 1)^2_{\odot}] \Big\} \nonumber\\
&& + \beta (J^{\delta}_{\mu}  - J_{\mu})^2.
\end{eqnarray}
Substituting (\ref{eq25}) into (\ref{eq_31}), we derive the
mean-variance combined performance difference formula between mixed
policy $d^{\delta, d'}$ and deterministic policy $d$ as below
\begin{equation}\label{eq_dif3}
J^{\delta}_{\mu,\sigma} - J_{\mu,\sigma} = \delta \bm \pi^{\delta}
\Big[ (\bm P' - \bm P)\bm g_{\mu,\sigma} + \bm r' - \beta (\bm r' -
J_{\mu} \bm 1)^2_{\odot} - \bm r + \beta (\bm r - J_{\mu} \bm
1)^2_{\odot} \Big] + \beta (J^{\delta}_{\mu}-J_{\mu})^2.
\end{equation}

Comparing (\ref{eq_dif3}) with (\ref{eq_dif2}), we can see that the
first parts of the right-hand-side of these two difference formulas
have a linear factor $\delta$. When $\delta = 1$,  we can see that
(\ref{eq_dif3}) becomes (\ref{eq_dif2}). With (\ref{eq_dif3}), we
further derive the following lemma about the \emph{performance
derivative formula} of the mean-variance combined metrics in the
\emph{mixed policy space}.
\begin{lemma}\label{lemma2}
For an MDP with the current deterministic policy $d$ and associated
$(\bm P, \bm r)$, if we choose any new deterministic policy $d'$
with $(\bm P',\bm r')$, the derivative of the mean-variance combined
performance $J_{\mu,\sigma}$ with respect to the mixed probability
$\delta$ is quantified by
\begin{equation}\label{eq_dev}
\frac{\dif J_{\mu,\sigma}}{\dif \delta} = \bm \pi \Big[ (\bm P' -
\bm P)\bm g_{\mu,\sigma} + \bm r' - \beta (\bm r' - J_{\mu} \bm
1)^2_{\odot} - \bm r + \beta (\bm r - J_{\mu} \bm 1)^2_{\odot}
\Big].
\end{equation}
\end{lemma}
\begin{proof}
With (\ref{eq_dif3}), taking the derivative operation with respect
to $\delta$ on both sides and letting $\delta \rightarrow 0$, we
directly obtain
\begin{eqnarray}
\frac{\dif J_{\mu,\sigma}}{\dif \delta} &=& \lim\limits_{\delta
\rightarrow 0}\left\{\bm \pi^{\delta} \Big[ (\bm P' - \bm P)\bm
g_{\mu,\sigma} + \bm r' - \beta (\bm r' - J_{\mu} \bm 1)^2_{\odot} -
\bm r + \beta (\bm r - J_{\mu} \bm 1)^2_{\odot} \Big] +
2\beta(J^{\delta}_{\mu} - J_{\mu})\frac{\dif J^{\delta}_{\mu} }{\dif
\delta}\right\} \nonumber\\
&=& \bm \pi \Big[ (\bm P' - \bm P)\bm g_{\mu,\sigma} + \bm r' -
\beta (\bm r' - J_{\mu} \bm 1)^2_{\odot} - \bm r + \beta (\bm r -
J_{\mu} \bm 1)^2_{\odot} \Big] + 2\beta\frac{\dif J_{\mu} }{\dif
\delta}\lim\limits_{\delta \rightarrow 0}\left\{J^{\delta}_{\mu} -
J_{\mu}\right\}\nonumber\\
&=& \bm \pi \Big[ (\bm P' - \bm P)\bm g_{\mu,\sigma} + \bm r' -
\beta (\bm r' - J_{\mu} \bm 1)^2_{\odot} - \bm r + \beta (\bm r -
J_{\mu} \bm 1)^2_{\odot} \Big], \nonumber
\end{eqnarray}
where we utilize the fact that $\lim\limits_{\delta \rightarrow
0}\bm \pi^{\delta} = \bm \pi$ and $\lim\limits_{\delta \rightarrow
0}(J^{\delta}_{\mu} - J_{\mu}) = 0$. The lemma is proved.
\end{proof}
Therefore, performance derivative formula (\ref{eq_dev}) can
directly quantify the performance gradient in the mixed policy space
along with a policy perturbation direction between any two
deterministic policies $d$ and $d'$.

Second, we focus on a more general form of randomized policies where
the policy is parameterized by $\bm \theta := (\theta_{i,a})$'s,
parameter $\theta_{i,a}$ indicates the probability of choosing
action $a \in \mathcal A$ at state $i \in \mathcal S$. We define the
randomized policy space as $\bm \Theta:=\{\mbox{all } \bm \theta |
\theta_{i,a}\geq 0, \sum_{a \in \mathcal A}\theta_{i,a}=1, \forall i
\in \mathcal S, \forall a \in \mathcal A\}$. Similarly, we consider
the performance difference of the mean-variance combined metrics
under two different randomized policies $\bm \theta$ and $\bm
\theta'$.

From the definition of $\theta_{i,a}$'s, we can derive that the
transition probabilities under $\bm \theta$ and $\bm \theta'$ are as
follows, respectively
\begin{equation}\label{eq_PP}
\begin{array}{l}
P^{\bm \theta}(i,j) := \sum_{a \in \mathcal A} p^a(i,j)
\theta_{i,a}, \quad i,j\in \mathcal S, \\
P^{\bm \theta'}(i,j) := \sum_{a \in \mathcal A} p^a(i,j)
\theta'_{i,a}, \quad i,j\in \mathcal S.
\end{array}
\end{equation}
The mean-variance combined cost function under policy $\bm \theta$
has the following definition
\begin{equation}\label{eq_ftheta1}
f^{\bm \theta}_{\mu,\sigma}(i) := \sum_{a\in\mathcal
A}\theta_{i,a}\left[ r(i,a) - \beta(r(i,a) - J^{\bm \theta}_{\mu})^2
\right], \quad i \in \mathcal S,
\end{equation}
where the mean performance $J^{\bm \theta}_{\mu}$ is defined as
\begin{equation}
J^{\bm \theta}_{\mu} := \sum_{i \in \mathcal S} \pi^{\bm
\theta}(i)\sum_{a \in \mathcal A} \theta_{i,a} r(i,a), \nonumber
\end{equation}
and $\pi^{\bm \theta}(i)$ is the steady state distribution at state
$i$ under policy $\bm \theta$. Similarly, we obtain the cost
function under policy $\bm \theta'$ as
\begin{equation}\label{eq_ftheta2}
f^{\bm \theta'}_{\mu,\sigma}(i) := \sum_{a\in\mathcal
A}\theta'_{i,a}\left[ r(i,a) - \beta(r(i,a) - J^{\bm
\theta'}_{\mu})^2 \right], \quad i \in \mathcal S,
\end{equation}
and
\begin{equation}
J^{\bm \theta'}_{\mu} := \sum_{i \in \mathcal S} \pi^{\bm
\theta'}(i)\sum_{a \in \mathcal A} \theta'_{i,a} r(i,a). \nonumber
\end{equation}
Similar to the derivation of the performance difference formula
(\ref{eq_dif}), we can also obtain the following difference formula
under two randomized policies $\bm \theta$ and $\bm \theta'$
\begin{equation}
J^{\bm \theta'}_{\mu,\sigma} - J^{\bm \theta}_{\mu,\sigma} = \sum_{i
\in \mathcal S}\pi^{\bm \theta'}(i)\left\{ f^{\bm
\theta'}_{\mu,\sigma}(i)  - f^{\bm \theta}_{\mu,\sigma}(i) + \sum_{j
\in \mathcal S}\big[P^{\bm \theta'}(i,j) - P^{\bm
\theta}(i,j)\big]g^{\bm \theta}_{\mu,\sigma}(j)  \right\},
\end{equation}
where $g^{\bm \theta}_{\mu,\sigma}(j)$ is the performance potential
defined in (\ref{eq_g2}) with $P^{\bm \theta}(i,j)$'s and $f^{\bm
\theta}_{\mu,\sigma}(i)$'s under policy $\bm \theta$. Substituting
(\ref{eq_PP}), (\ref{eq_ftheta1}), and (\ref{eq_ftheta2}) into the
above equation, we can further derive the following formula
\begin{eqnarray}\label{eq_dif4}
J^{\bm \theta'}_{\mu,\sigma} - J^{\bm \theta}_{\mu,\sigma} &=&
\sum_{i \in \mathcal S}\pi^{\bm \theta'}(i)\sum_{a \in \mathcal
A}(\theta'_{i,a}-\theta_{i,a})\left\{\sum_{j \in \mathcal
S}p^{a}(i,j)g^{\bm \theta}_{\mu,\sigma}(j) + r(i,a) - \beta r^2(i,a)
+ 2 \beta J^{\bm \theta}_{\mu}r(i,a) \right\} \nonumber\\
&& + \beta (J^{\bm \theta'}_{\mu} - J^{\bm \theta}_{\mu})^2.
\end{eqnarray}
Therefore, with (\ref{eq_dif4}), we directly derive the following
lemma about the mean-variance combined performance derivative with
respect to the randomized policy.
\begin{lemma}\label{lemma3}
For a randomized MDP with parameterized policy $\bm \theta$, where
$\theta_{i,a}$ indicates the probability of choosing action $a$ at
state $i$, the derivative of the mean-variance combined performance
with respect to the parameter $\bm \theta$ is quantified by
\begin{equation}\label{eq_devpara}
\frac{\dif J^{\bm \theta}_{\mu,\sigma}}{\dif \theta_{i,a}} =
\pi^{\bm \theta}(i) \left\{\sum_{j \in \mathcal S}p^{a}(i,j)g^{\bm
\theta}_{\mu,\sigma}(j) + r(i,a) - \beta r^2(i,a) + 2 \beta J^{\bm
\theta}_{\mu}r(i,a) \right\}, \quad i \in \mathcal S, a \in \mathcal
A.
\end{equation}
\end{lemma}
\begin{proof}
Taking the derivative operation with respect to $\theta_{i,a}$ on
both sides of (\ref{eq_dif4}) and letting $\bm \theta' \rightarrow
\bm \theta$, we have
\begin{eqnarray}
\frac{\dif J^{\bm \theta}_{\mu,\sigma}}{\dif \theta_{i,a}}
\hspace{-0.2cm}&=&\hspace{-0.2cm} \lim\limits_{\bm \theta'
\rightarrow \bm \theta} \pi^{\bm \theta'}(i) \left\{\sum_{j \in
\mathcal S}p^{a}(i,j)g^{\bm \theta}_{\mu,\sigma}(j) + r(i,a) - \beta
r^2(i,a) + 2 \beta J^{\bm \theta}_{\mu}r(i,a) \right\} + 2\beta
\frac{\dif J^{\bm \theta}_{\mu}}{\dif \theta_{i,a}} \lim\limits_{\bm
\theta' \rightarrow \bm \theta} \left\{J^{\bm \theta'}_{\mu} -
J^{\bm
\theta}_{\mu}\right\} \nonumber\\
&=& \pi^{\bm \theta}(i) \left\{\sum_{j \in \mathcal
S}p^{a}(i,j)g^{\bm \theta}_{\mu,\sigma}(j) + r(i,a) - \beta r^2(i,a)
+ 2 \beta J^{\bm \theta}_{\mu}r(i,a) \right\}, \nonumber
\end{eqnarray}
where we utilize the fact that $\lim\limits_{\bm \theta' \rightarrow
\bm \theta} \pi^{\bm \theta'}(i) = \pi^{\bm \theta}(i)$ and
$\lim\limits_{\bm \theta' \rightarrow \bm \theta} \left\{J^{\bm
\theta'}_{\mu} - J^{\bm \theta}_{\mu}\right\} = 0$. The lemma is
proved.
\end{proof}

Based on (\ref{eq_devpara}), we can further develop policy gradient
algorithms to find the optimal $\bm \theta^*$. With
Lemma~\ref{lemma3}, we derive the following theorem about the
optimality of deterministic policies.

\begin{theorem}\label{theorem_deterministic}
For the MDP optimization problem of mean-variance combined metrics
defined in (\ref{eq_prob}), a deterministic policy can achieve the
optimal value.
\end{theorem}
\begin{proof}
We choose any randomized policy defined by parameters
$\theta_{i,a}$, $i \in \mathcal S$, $a \in \mathcal A$. The
associated transition probability and cost function are defined by
(\ref{eq_PP}) and (\ref{eq_ftheta1}), respectively. With
(\ref{eq_dif4}) and the necessary condition in
Theorem~\ref{theorem2}, we directly derive the following result. For
any $i \in \mathcal S$, if $\bm \theta^*_i :=
(\theta^*_{i,a_1},\theta^*_{i,a_2},\cdots,\theta^*_{i,a_A})$ is
optimal, it must satisfy the following conditions
\begin{equation}\label{eq_lp}
\begin{array}{ll}
\bm \theta^*_i = & \argmax\limits_{\bm \theta_i} \sum_{a \in
\mathcal A}\theta_{i,a} \left\{\sum_{j \in \mathcal
S}p^{a}(i,j)g^{\bm \theta^*}_{\mu,\sigma}(j) + r(i,a) - \beta
r^2(i,a) + 2 \beta J^{\bm
\theta^*}_{\mu}r(i,a) \right\}, \\
\mbox{s.t.,} & \sum_{a \in \mathcal A} \theta_{i,a} = 1,\\
\quad & \theta_{i,a} \geq 0, \quad \forall a \in \mathcal A.
\end{array}
\end{equation}
In the above problem, the values in the braces are given and the
$\theta_{i,a}$'s are optimization variables. Obviously,
(\ref{eq_lp}) is a linear program. With the well known results of
linear programming, an optimal solution $\theta^*_{i,a}$ can be
found on the vertexes of the multidimensional polyhedron composed by
the value domain of $\theta_{i,a}$'s \citep{Chong13}. We further
look at the constraints defined in (\ref{eq_lp}). It is easy to find
that the feasible domain of $\theta_{i,a}$'s is $[0, 1]$. Therefore,
the vertexes are either 0 or 1, so are the optimal solution
$\theta^*_{i,a}$'s. This indicates that the optimal policy can be
deterministic and the theorem is proved.
\end{proof}

\noindent\textbf{Remark~3.} For many risk-aware MDPs, a
deterministic policy cannot achieve the optimal value
\citep{Chung94}. This is partly because that those risk-aware MDPs
have a constrained form. It is well known that constrained MDPs may
not achieve optimum at deterministic policies \citep{Altman99}.
However, our optimization problem defined in (\ref{eq_prob}) is not
of a constrained form, which partly supports the result in
Theorem~\ref{theorem_deterministic}.

With the closed-form solution of gradients represented by
(\ref{eq_dev}) or (\ref{eq_devpara}), we can further develop policy
gradient-based algorithms to solve our optimization problem
(\ref{eq_prob}). Gradient-based algorithms are widely adopted in the
community of reinforcement learning \citep{Prashanth13,Tamar12}. It
is worth noting that Lemma~\ref{lemma3} is about a special case of
parameterized policy $\bm \theta$, where $\theta_{i,a}$ indicates
the probability of choosing action $a \in \mathcal A$ at state $i
\in \mathcal S$. Thus, the dimension of such parameterized policy is
$|\bm \theta| = |\mathcal S||\mathcal A| = SA$. In general, the
space of parameters $\bm \theta$ can have a much lower dimension
than the original policy space mapping from $\mathcal S$ to
$\mathcal A$. For example, $\bm \theta$ can be the weights of kernel
functions or the parameters of a neural network, such as the actor
network (or policy network) widely used in deep reinforcement
learning. We can also develop similar gradients like
(\ref{eq_devpara}). However, such gradient-based method usually
suffers from intrinsic deficiencies, such as being trapped into a
local optimum, difficulty of selecting learning step sizes. In the
next section, we will further develop a policy iteration type
algorithm to solve (\ref{eq_prob}), which usually has a fast
convergence speed in practice.

\section{Optimization Algorithm}\label{section_algo}
In this section, we develop an iterative algorithm to solve the
optimization problem (\ref{eq_prob}) based on the performance
difference formula (\ref{eq_dif2}). This algorithm is of a policy
iteration type. We also prove that the algorithm can converge to a
local optimum in both the mixed and the randomized policy spaces.

The performance difference formula (\ref{eq_dif2}) and
Theorem~\ref{theorem1} directly indicate an approach to generate
improved policies. Therefore, we can develop an iterative procedure
to optimize the system performance of mean-variance combined
metrics, which is stated in Algorithm~\ref{algo1}.

\begin{algorithm}[htbp]
  \caption{An iterative algorithm to find the optimal mean-variance combined metric.}\label{algo1}
  \begin{algorithmic}[1]

\State arbitrarily choose an initial policy $d^{(0)} \in \mathcal D$
and set $l=0$;

\Repeat

\State for the current policy $d^{(l)}$, compute or estimate the
values of $J_{\mu}$, $J_{\mu,\sigma}$, and $\bm g_{\mu,\sigma}$
based on their definitions (\ref{eq_J}), (\ref{eq_Jmuvar}), and
(\ref{eq_g}), respectively;

\State generate a new policy $d^{(l+1)}$ as follows:
\begin{equation}\label{eq_PI}
d^{(l+1)}(i) := \argmax\limits_{a \in \mathcal A}\Bigg\{ r(i,a) -
\beta[r(i,a)-J_{\mu}]^2 + \sum_{j \in \mathcal S}p^a(i,j)
g_{\mu,\sigma}(j) \Bigg\}
\end{equation}
for all $i \in \mathcal S$. While breaking ties to avoid policy
oscillations, we ensure $d^{(l+1)}(i) = d^{(l)}(i)$ if possible (if
$d^{(l)}(i)$ can already achieve $\max$ in (\ref{eq_PI}));

\State set $l := l+1$;

\Until{$d^{(l)} = d^{(l-1)}$}

\Return $d^{(l)}$.

\end{algorithmic}
\end{algorithm}

As we discussed at the end of Section~\ref{section_model}, because
the cost function (\ref{eq_f}) is not additive, our optimization
problem (\ref{eq_prob}) is not a standard MDP problem and the
dynamic programming is not applicable. Algorithm~\ref{algo1} treats
the original problem (\ref{eq_prob}) as if it is a standard MDP with
new cost function $r(i,a)-\beta[r(i,a)-J_{\mu}]^2$, where $J_{\mu}$
is a constant at the current iteration and will be updated at next
iterations. With Theorem~\ref{theorem1}, we can see that the new
policy generated by (\ref{eq_PI}) is better than the current policy.
Therefore, the policy will be improved continually in
Algorithm~\ref{algo1}. We can also observe that
Algorithm~~\ref{algo1} is of a policy iteration type. However, the
global convergence of the traditional policy iteration cannot be
directly extended to Algorithm~\ref{algo1} since our problem
(\ref{eq_prob}) is not a standard MDP. Below, we first give
definitions of local optimum in the mixed policy space and the
randomized policy space, respectively. Then, we study the
convergence property of Algorithm~\ref{algo1}.
\begin{definition}\label{def1}
For a policy $d \in \mathcal D$, if there exists $\Delta \in (0,1)$,
we always have $J^{d}_{\mu,\sigma} \geq
J^{d^{\delta,d'}}_{\mu,\sigma}$ for any $d' \in \mathcal D$ and
$\delta \in (0,\Delta)$ , then we say $d$ is a local optimal policy
in the mixed policy space.
\end{definition}
If we extend our policy space to the randomized policy space with
parameters $\bm \theta$, we can also have the following definition
of local optimal policy.
\begin{definition}\label{def2}
For a randomized policy $\bm \theta \in \bm \Theta$, if there exists
$\Delta > 0$, we always have $J^{\bm \theta}_{\mu,\sigma} \geq
J^{\bm \theta'}_{\mu,\sigma}$ for any $\bm \theta' \in \bm \Theta$
and $||\bm \theta - \bm \theta'|| < \Delta$ where $||\cdot||$ can be
an Euclidean distance, then we say $\bm \theta$ is a local optimal
policy in the randomized policy space.
\end{definition}

Since $\mathcal D$ has a finite set of policies, we need the
randomization to make the space continuous such that a local optimum
can be defined. Definition~\ref{def1} indicates that if $d$ is a
local optimum, then it is not worse than any mixed policy in its
small enough neighborhood, along with any perturbation direction
from $d$ to $d' \in \mathcal D$. In other words, if $d$ is a local
optimum in the mixed policy space, then we always have $\frac{\dif
J^d_{\mu,\sigma}}{\dif \delta} \leq 0$ along with any perturbation
direction. Definition~\ref{def2} is more natural since it is defined
in the fully randomized policy space, which is continuous.

With Definitions~\ref{def1} \& \ref{def2} , we further derive the
following theorem about the local optimal convergence of
Algorithm~\ref{algo1}.
\begin{theorem}\label{theorem3}
Algorithm~\ref{algo1} converges to a local optimum, both in the
mixed policy space and the randomized policy space.
\end{theorem}
\begin{proof}
First, we prove the convergence of Algorithm~\ref{algo1}. From the
policy improvement step in (\ref{eq_PI}), we can see that the newly
generated policy $d^{(l+1)}$ is not worse than $d^{(l)}$, based on
the result of Theorem~\ref{theorem1}. More specifically, if
$d^{(l+1)} \neq d^{(l)}$, we can see that for at least one state
$i$, the $\max$ operator in (\ref{eq_PI}) is achieved and its value
in the braces of (\ref{eq_PI}) is strictly larger than the
associated value with the current action $d^{(l)}(i)$. Therefore,
with the last part of Theorem~\ref{theorem1}, we have
$J^{d^{(l+1)}}_{\mu,\sigma}
> J^{d^{(l)}}_{\mu,\sigma}$ and the newly generated policy is
strictly improved. Since the policy space $\mathcal D$ is finite,
Algorithm~\ref{algo1} will stop after a finite number of iterations.
Thus, the convergence of Algorithm~\ref{algo1} is proved.

Second, we prove that the convergence point is a local optimum of
the mixed policy space in Definition~\ref{def1}. From
Algorithm~\ref{algo1}, we can see that when the algorithm stops we
cannot find a new different policy generated by (\ref{eq_PI}). In
other words, any other policy with $(\bm P', \bm r')$ cannot make
the term in the braces of (\ref{eq_PI}) strictly bigger than that of
the current policy with $(\bm P, \bm r)$. That is, when
Algorithm~\ref{algo1} stops, we have
\begin{equation}\label{eq46}
r(i) - \beta[r(i)-J_{\mu}]^2 + \sum_{j \in \mathcal S}p(i,j)
g_{\mu,\sigma}(j) \geq r'(i) - \beta[r'(i)-J_{\mu}]^2 + \sum_{j \in
\mathcal S}p'(i,j) g_{\mu,\sigma}(j), \quad \forall i \in \mathcal
S,
\end{equation}
for any policy with $(\bm P', \bm r')$. Substituting the above
inequality into the derivative formula (\ref{eq_dev}) and using the
fact that the elements of $\bm \pi$ are always positive, we have
\begin{equation}
\frac{\dif J_{\mu,\sigma}}{\dif \delta} \leq 0, \nonumber
\end{equation}
along with any policy perturbation direction in the mixed policy
space. Therefore, with Definition~\ref{def1} and the first order of
Taylor expansion of $J_{\mu,\sigma}$ with respect to $\delta$, we
can see that the policy with $(\bm P, \bm r)$ converged to is a
local maximum in the mixed policy space.

Third, we further prove that the convergence point is also a local
optimum of the randomized policy space in Definition~\ref{def2}.
Suppose Algorithm~\ref{algo1} stops at policy $d$ with $(\bm P, \bm
r)$, we have (\ref{eq46}) and rewrite it as
\begin{equation}
r(i) - \beta r^2(i) + 2\beta J_{\mu} r(i) + \sum_{j \in \mathcal
S}p(i,j) g_{\mu,\sigma}(j) \geq r'(i) - \beta r'^2(i) + 2\beta
J_{\mu}r'(i) + \sum_{j \in \mathcal S}p'(i,j) g_{\mu,\sigma}(j),
\nonumber
\end{equation}
for any $(\bm P', \bm r')$ and $i \in \mathcal S$. In other words,
at each state $i$, the current policy $d$ always has the maximal
value of $r(i,a) - \beta r^2(i,a) + 2\beta J_{\mu} r(i,a) + \sum_{j
\in \mathcal S}p^a(i,j) g_{\mu,\sigma}(j)$ over all actions $a \in
\mathcal A$, i.e.,
\begin{equation}
d(i) = \argmax\limits_{a \in \mathcal A}\left\{ r(i,a) - \beta
r^2(i,a) + 2\beta J_{\mu} r(i,a) + \sum_{j \in \mathcal S}p^a(i,j)
g_{\mu,\sigma}(j) \right\}, \quad  i \in \mathcal S. \nonumber
\end{equation}
Substituting the above equation into the policy gradient
(\ref{eq_devpara}) in Lemma~\ref{lemma3}, we can see that the
converged policy $d$ (also denoted by $\bm \theta$ in the randomized
policy space) has
\begin{equation}\label{eq50}
\frac{\dif J^{\bm \theta}_{\mu,\sigma}}{\dif \theta_{i,a}} \geq
\frac{\dif J^{\bm \theta}_{\mu,\sigma}}{\dif \theta_{i,a'}}, \qquad
a=d(i), \ \forall a' \in \mathcal A.
\end{equation}
That is, at each state $i$, the converged policy $d$ always has the
maximal gradient at $\theta_{i,a}$ than other gradients at
$\theta_{i,a'}$, $a' \in \mathcal A$. Since the current policy $d$
is deterministic and it has $\theta_{i,a} = 1$ for $a = d(i)$ and
$\theta_{i,a'} = 0$ for other $a' \in \mathcal A$, we perturb in a
small neighborhood of $\bm \theta$. More specifically,
$\theta_{i,a}$ is perturbed from $1$ to $1 - \Delta$,
$\theta_{i,a'}$'s are perturbed from $0$ to $\Delta^{a'}$, and it
must satisfy $\sum_{a' \neq a} \Delta^{a'} = \Delta$. We have
\begin{eqnarray}
J^{\bm \theta}_{\mu,\sigma} - J^{\bm \theta+\Delta}_{\mu,\sigma} &=&
\frac{\dif J^{\bm \theta}_{\mu,\sigma}}{\dif \theta_{i,a}} \Delta -
\sum_{a' \in \mathcal A, a' \neq a} \frac{\dif J^{\bm
\theta}_{\mu,\sigma}}{\dif \theta_{i,a'}}
\Delta^{a'} + o(\Delta) \nonumber\\
&=& \sum_{a' \in \mathcal A, a' \neq a} \left[ \frac{\dif J^{\bm
\theta}_{\mu,\sigma}}{\dif \theta_{i,a}} - \frac{\dif J^{\bm
\theta}_{\mu,\sigma}}{\dif \theta_{i,a'}}\right]\Delta^{a'} +
o(\Delta). \nonumber
\end{eqnarray}
Substituting (\ref{eq50}) into the above equation, we always have
$J^{\bm \theta}_{\mu,\sigma} - J^{\bm \theta+\Delta}_{\mu,\sigma}
\geq 0$ for any small enough neighborhood of $\bm \theta$. This
analysis procedure is valid for each state $i \in \mathcal S$.
Therefore, the converged policy $d$ is a local optimum in the
randomized policy space. The theorem is proved.
\end{proof}

Because of the quadratic form of the variance related metrics, our
optimization problem (\ref{eq_prob}) usually is a multi-modal
function in the policy space, which is hard to find the global
optimum. This is also one of the reasons that Algorithm~\ref{algo1}
may only converge to a local optimum. Similar to the condition
discussed in Remark~2, we also have the following remark to clarify
the algorithm's global convergence.

\noindent\textbf{Remark~4.} If the system mean reward is the same
for all policy $d \in \mathcal D$, i.e., $J^d_{\mu}$ is independent
of $d$, Algorithm~\ref{algo1} will converge to the global optimum.

One example satisfying the condition in Remark~4 can be found in the
experiment of the next section, where the global convergence is
guaranteed. Although Algorithm~\ref{algo1} may only converge to a
local optimum, it has a form of policy iteration, which is similar
to the classical policy iteration in the traditional MDP theory.
Therefore, it is expected that Algorithm~\ref{algo1} has a similar
convergence behavior as that of the classical policy iteration.
However, a specific analysis of the algorithmic complexity of
Algorithm~\ref{algo1} is difficult to derive. It is because the
algorithmic complexity of the classical policy iteration is still an
open problem \citep{Littman95}. Nevertheless, it is often observed
that the policy iteration converges very fast. With the difference
formula (\ref{eq_dif}), we can see that each iteration in
(\ref{eq_PI}) will strictly improve the system performance. However,
the gradient-based approach has to carefully select step sizes,
otherwise it may jump to a worse policy. Such merits of the policy
iteration are desirable for policy gradient algorithms. For example,
the proximal policy optimization (PPO) is an efficient reinforcement
learning algorithm emerging in recent years \citep{Schulman15}. It
outperforms the traditional policy gradient algorithms in
large-scale reinforcement learning problems, such as optimizing
policy to drive a complex game called Dota~2 in the project OpenAI
Five. The PPO algorithm can also be viewed as an attempt to use
approximated policy iteration to guarantee a strict improvement of
each policy update, compared with the policy gradient method.
Therefore, it is reasonable to argue that Algorithm~\ref{algo1} also
has a fast convergence speed for small or medium size risk-sensitive
MDP problems.

For large-scale risk-sensitive MDP problems, we can utilize the
widely used approximation techniques to improve the performance of
Algorithm~\ref{algo1}, such as neuro-dynamic programming
\citep{Bertsekas96}, approximate dynamic programming
\citep{Powell07,Yu17}, deep neural networks \citep{Silver16}, and
other data-driven learning techniques. We may call it
\emph{risk-sensitive reinforcement learning}
\citep{Borkar2010,Huang17,Prashanth13}, which is still a new
research direction of reinforcement learning deserving further
investigations. Interesting topics may include the efficient
estimation of key quantities $\bm g_{\mu,\sigma}$, the robustness
analysis of algorithms with respect to the inaccurate values of $\bm
P$ or $\bm g_{\mu,\sigma}$ (see robust MDPs by
\cite{Chow15,Lim13,Nilim05}), etc.

\section{Numerical Experiments} \label{section_expriment}
In this section, we use an example about the fluctuation reduction
of wind power in energy storage systems to demonstrate the
applicability of our optimization method. The parameter setting and
engineering constraints have been largely simplified such that the
key idea of this example is concise and easy to follow.

\subsection{Wind Abandonment Not Allowed}
We consider a wind farm with a battery energy storage system
connected to the main power grid. Note that all the continuous
variables in this problem are discretized properly. The wind power
is assumed as a stationary stochastic process. We use a Markov chain
$\bm X := \{X_t\}$ to model the dynamics of wind power, where $X_t$
is the wind power at time epoch $t$ (hourly in this paper),
$t=1,2,\cdots$. The transition probability matrix of $X_t$ is
denoted as $\bm P$, which can be estimated from statistics
\citep{Luh14,Yang18}. The battery storage has a capacity $B$ and
$b_t$ denotes the remaining battery energy level at time $t$. The
system state is defined as $(X_t,b_t)$. The action is denoted as
$A_t$ which indicates the charging or discharging power of the
battery at time $t$. The battery has a maximal charging and
discharging power and the value domain of $A_t$ is $\mathcal A :=
\{-2,-1,0,1,2\}$. Positive element of $\mathcal A$ indicates
discharging power and negative one indicates charging power. When we
select $A_t$, it should be constrained by the remaining capacity of
the battery, i.e., $b_t - B\leq A_t \leq b_t$. If $A_t$ is adopted,
the battery energy level at the next time epoch is updated as
$b_{t+1} = b_{t} - A_t$.

The parameter setting of this problem is summarized in the following
tables. The number of wind power states is 6, and the correspondence
between the state and the wind power is shown in Table~\ref{tab1}.
The battery capacity is $B=5$, and the correspondence between the
battery states and battery energy level is shown in
Table~\ref{tab2}. Table~\ref{tab3} shows the actions and their
corresponding operations of the battery energy storage system.

\begin{table}[htbp]
    \centering
    \caption{{\small States of the wind power output}}\label{tab1}
    \begin{tabular}{ccccccc}
        \hline
        \toprule 
        State & 1 & 2 & 3 & 4 & 5 & 6\\ \hline
        Wind power/MW & 0 & 1 & 2 & 3 & 4 & 5 \\ \hline
    \end{tabular}
\end{table}

\begin{table}[htbp]
    \centering
    \caption{{\small States of the battery energy level}}\label{tab2}
    \begin{tabular}{ccccccc}
        \hline
        \toprule 
        State & 1 & 2 & 3 & 4 & 5 & 6\\ \hline
        Battery energy level/MWh & 0 & 1 & 2 & 3 & 4 & 5 \\ \hline
    \end{tabular}
\end{table}

\begin{table}[htbp]
    \centering
    \caption{{\small Scheduling actions of the battery}}\label{tab3}
    \begin{tabular}{cccccc}
        \hline
        \toprule 
        Action  & $-2$ & $-1$ & 0 & 1 & 2 \\ \hline
        Battery (dis)charging power/MW & $-2$ & $-1$ & 0 & $+1$ & $+2$ \\ \hline
    \end{tabular}
\end{table}

The wind state transition probability matrix $\bm P$ is calculated
based on the real data provided by the Measurement and
Instrumentation Data Center (MIDC) in the National Renewable Energy
Laboratory \citep{NREL}, in which the wind speed is measured since
1996.
\begin{equation}\label{eq_52}
\bm{P}= \left(
\begin{matrix}
0.53 & 0.18 & 0.19 & 0.04 & 0.01 & 0.05\\
0.51 & 0.08 & 0.20 & 0.08 & 0.02 & 0.11\\
0.35 & 0.11 & 0.19 & 0.11 & 0.03 & 0.21\\
0.27 & 0.15 & 0.15 & 0.14 & 0.03 & 0.26\\
0.14 & 0.11 & 0.13 & 0.15 & 0.05 & 0.42\\
0.09 & 0.03 & 0.06 & 0.06 & 0.03 & 0.73
\end{matrix}
\right).
\end{equation}
It is easy to verify that the ergodicity in
Assumption~\ref{assumption1} is satisfied in this experiment
according to the value of the transition probability matrix
(\ref{eq_52}) and the feasible actions.

In this subsection, we assume that the wind abandonment is not
allowed, i.e., all the wind power generated should go to either the
battery or the grid. The total output power of the system is denoted
as $Y_t$ and we have $Y_t = X_t + A_t$. The control policy $d$ of
the battery is a mapping from state $(X_t,b_t)$ to action $A_t$,
i.e., the charging or discharging power of the battery is determined
by $A_t = d(X_t,b_t)$. The optimization objective includes two
parts. One is the average power generated $\mathbb{E}[Y_t]$, which
reflects the economic benefit of the system. We can also further
include other economic metrics if needed, such as the operating cost
of the battery system. The other is the fluctuation of the output
power $\sigma^2[Y_t]$, which reflects the power quality or system
safety. If $\sigma^2[Y_t]$ is smaller, it indicates that the output
power is more stable and the power quality is better. The
coefficient $\beta$ can be viewed as a shadow price of the power
quality or safety for the grid system, which has been attracting
more attention by grid company recently \citep{Li14}. We aim at
maximizing the average output power $\mathbb{E}[Y_t]$ while reducing
the power fluctuation $\sigma^2[Y_t]$. Such an objective is exactly
a combined metric of mean and variance. Thus, we can apply the
approach proposed in this paper to study this fluctuation reduction
problem of renewable energy with storage systems. An illustrative
diagram of this optimization problem is shown in
Fig.~\ref{fig_wind_n}.

\begin{figure}[htbp]
\centering
\includegraphics[width=0.9\columnwidth]{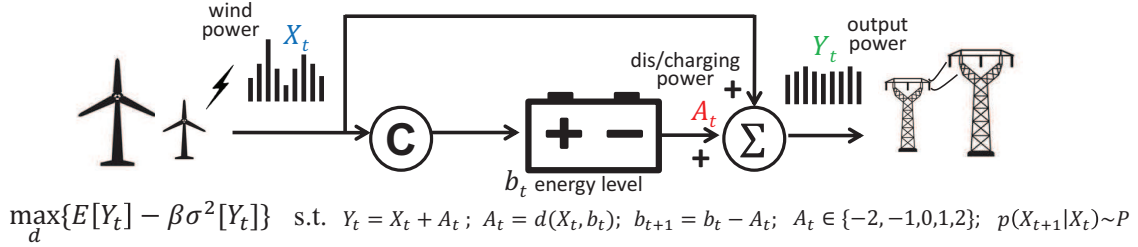}
\caption{The control of renewable energy with storage systems
without wind abandonment.}\label{fig_wind_n}
\end{figure}

We apply Algorithm~\ref{algo1} to find the optimal control policy of
the battery. Note that since the wind abandonment is not allowed,
the long-run average output power $\mathbb{E}[Y_t]$ is not affected
by actions. That is, the value of $\mathbb{E}[Y_t]$ is independent
of policies and it is determined by the statistics of wind power.
Therefore, this scheduling problem is a special case of our
optimization problem (\ref{eq_prob}) for mean-variance combined
metrics and it satisfies the condition discussed in Remarks~2\&4.
The necessary condition in Theorem~\ref{theorem2} is also a
sufficient condition for this case and Algorithm~\ref{algo1}
converges to the global optimum, as indicated by Remark~4. For
demonstration, we choose $\beta=0.1$ and the optimization results
are illustrated by Fig.~\ref{fig_ite1}. From Fig.~\ref{fig_ite1}, we
can see that the variance $\sigma^2[Y_t]$ is continually reduced
while $\mathbb{E}[Y_t]$ remains unvaried as 2.3605. The
mean-variance combined metric $\mathbb{E}[Y_t] - \beta
\sigma^2[Y_t]$ is continually improved until converged after 4
iterations and its optimum is 2.0826.

\begin{figure}[htbp]
    \centering
    \includegraphics[width=0.55\columnwidth]{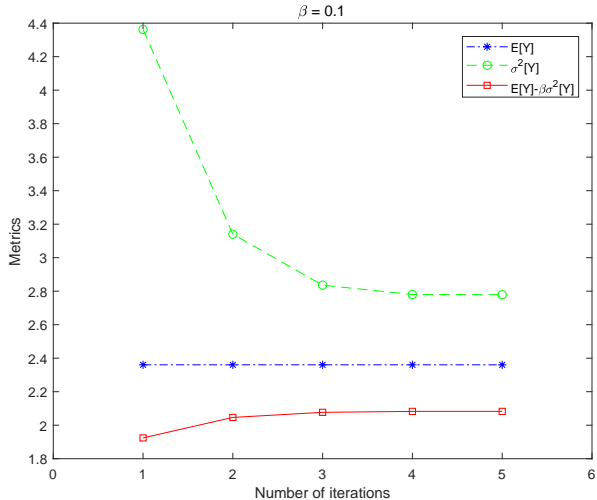}
    \caption{The convergence procedure of Algorithm~\ref{algo1} without wind abandonment.}\label{fig_ite1}
\end{figure}

Moreover, we randomly choose different initial policies and find
that Algorithm~\ref{algo1} always converges to the same optimal
policy, which is truly the global optimum. This is also consistent
with the results in Remarks~2\&4. The convergence results under some
different initial policies are illustrated in
Fig.~\ref{fig_pareto1}, where different colors and notations
indicate the convergence procedure from different initial policies.
With Fig.~\ref{fig_pareto1}, we can see that the algorithm converges
to the same point with the minimal variance, i.e., $\sigma^2[Y_t] =
2.7793$.

\begin{figure}[htbp]
    \centering
    \includegraphics[width=0.55\columnwidth]{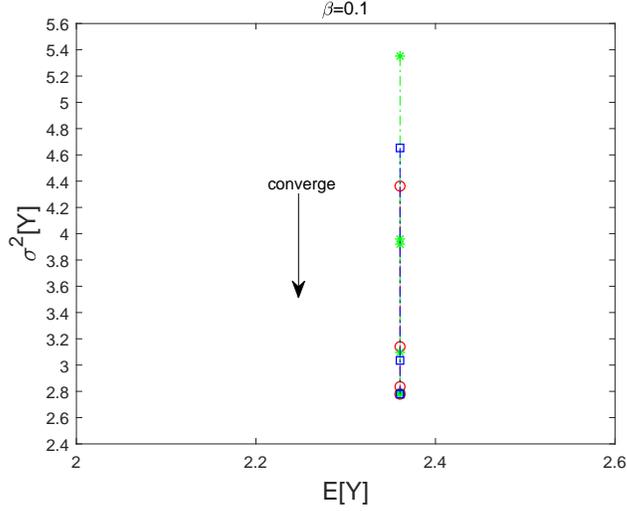}
    \caption{The same convergence point of Algorithm~\ref{algo1} under different initial policies, where Remarks~2\&4 work.}\label{fig_pareto1}
\end{figure}

To demonstrate the convergence efficiency of Algorithm~\ref{algo1},
we further conduct an experiment of traditional gradient-based
algorithm as a comparison. Since the control variable is the
charging or discharging power, it is denoted as $a \in \mathcal A =
\{-2,-1,0,1,2\}$. We use the randomized policy $\bm \theta$ such
that the parameters can be updated with gradient-descent algorithms.
With the performance derivative formula (\ref{eq_devpara}), we can
compute the gradient $\frac{\dif J_{\mu,\sigma}}{\dif \theta_{i,a}}$
for each $a \in \mathcal A$ and each system state $i=(X_t, b_t)$,
where $\theta_{i,a}$ is the probability of selecting action $a$ at
state $i$. We first choose an arbitrary initial policy, say $\bm
\theta_{i} = (0, 0, 1, 0, 0)$, i.e., the initial action is always
$a=0$. At each state $i \in \mathcal S$, after computing the value
of gradient $\frac{\dif J_{\mu,\sigma}}{\dif \theta_{i,a}}$, we find
the action $a_m$ whose gradient is maximal. Then, we update the
parameters as $\theta^{l+1}_{i,a_m} = \theta^{l}_{i,a_m} + \alpha_l$
and $\theta^{l+1}_{i,a} = \theta^{l}_{i,a}$ for all other $a \neq
a_m$, where $\alpha_l$ is the step size at the $l$th iteration and
we set it as a relatively large amount, say $\alpha_l = 1/\sqrt{l}$.
Then, we do normalization $\theta^{l+1}_{i,a} = \theta^{l+1}_{i,a} /
\sum_{a \in \mathcal A} \theta^{l+1}_{i,a}$ such that the sum of
$\bm \theta^{l+1}_{i}$ equals 1. We update the parameters $\bm
\theta$ repeatedly until the difference of $\bm \theta$ between two
successive iterations is smaller than a given threshold, say $0.1\%$
in ratio. The algorithm stops and outputs the current $\bm \theta^*$
as the optimal randomized policy. The experiment results are
illustrated by Fig.~\ref{fig_gradient}.

\begin{figure}[htbp]
    \centering
    \includegraphics[width=0.6\columnwidth]{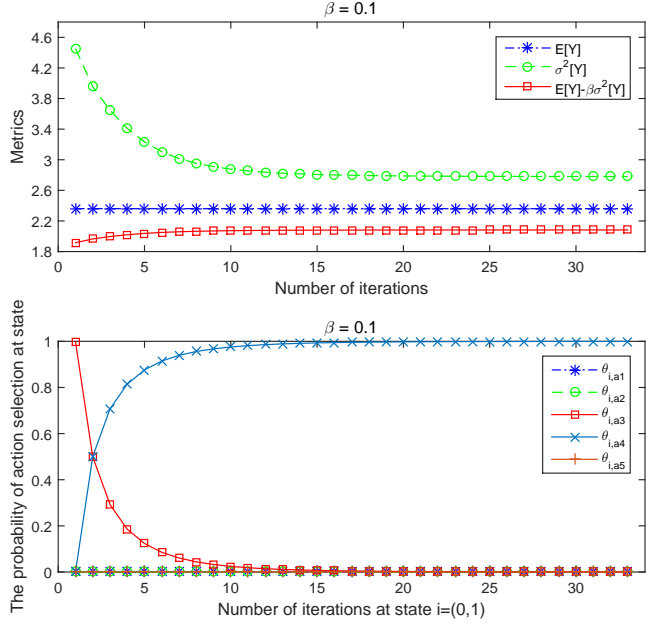}
    \caption{The experiment results of the gradient-based algorithm.}\label{fig_gradient}
\end{figure}

The two sub-figures in Fig.~\ref{fig_gradient} illustrate the
convergence curves of performance metrics and randomized parameters,
respectively. Here we only choose $\bm \theta_{i}$ as an example,
where $i=(0,1)$ means that the wind power is 0 and the battery
energy level is 1. From Fig.~\ref{fig_gradient} we can see that the
gradient-descent algorithm converges after 33 iterations, although
we choose a relatively large step size as $\alpha_l = 1/\sqrt{l}$.
At the convergence point, the performance metrics are
$\mathbb{E}[Y_t]=2.3605$, $\sigma^2[Y_t]=2.7802$, and
$\mathbb{E}[Y_t] - \beta \sigma^2[Y_t]=2.0825$, which are close to
and slightly worse than our previous results of
Algorithm~\ref{algo1}. We can also see that the convergence speed of
this gradient-based algorithm is much slower than that of our policy
iteration type algorithm. Moreover, from the convergence curves of
randomized parameters $\bm \theta_{i}$ in Fig.~\ref{fig_gradient},
we can see that $\bm \theta_{i}$ converges to a deterministic action
$(0,0,0,1,0)$. It means that the optimal action is discharging the
battery with power $a=1$ when the state is $i=(0,1)$, which looks
reasonable. This also verifies the optimality of deterministic
policies, as stated in Theorem~\ref{theorem_deterministic}.

\subsection{Wind Abandonment Allowed}
In this subsection, we study another scenario in which the wind
abandonment is allowed. All the parameter settings are the same as
those in the previous subsection. The difference is that we can
abandon an extra power $V_t$ if necessary. In this scenario, we
define the action as $U_t:=A_t-V_t$, where $0 \leq V_t \leq X_t$.
Variable $A_t$ has the same constraint as that in the previous
subsection, i.e., $A_t \in \mathcal A$ and $b_t - B\leq A_t \leq
b_t$. Although $A_t$ and $V_t$ are both variables, we can only use
$U_t$ as the decision variable after using the following reasonable
assumptions. If $0 \leq U_t \leq \min\{\max(\mathcal A), \ b_t\}$,
then we have $V_t=0$ and $A_t=U_t$, which means that the battery is
discharging and the wind power should not be abandoned. If
$\max\{\min(\mathcal A), \ b_t-B\} \leq U_t < 0$, then we have
$V_t=0$ and $A_t=U_t$, which means that the battery has potential
charging capacity unused and the wind power should not be abandoned.
If $-X_t \leq U_t < \max\{\min(\mathcal A), \ b_t-B\}$, then we have
$V_t=A_t-U_t$ and $A_t=\max\{\min(\mathcal A), \ b_t-B\}$, which
means that the battery is using full charging power capacity and
some wind power is abandoned.

In summary, the value domain of the decision variable $U_t$ is $-X_t
\leq U_t \leq \min\{\max(\mathcal A), \ b_t\}$ and it has the rule
as follows: if $\max\{\min(\mathcal A), \ b_t-B\} \leq U_t \leq
\min\{\max(\mathcal A), \ b_t\}$, then $V_t=0$ and $A_t=U_t$; if
$-X_t \leq U_t < \max\{\min(\mathcal A), \ b_t-B\}$, then
$V_t=A_t-U_t$ and $A_t=\max\{\min(\mathcal A), \ b_t-B\}$.
Therefore, we only need to optimize the variable $U_t$. Variables
$A_t$ and $V_t$ can be determined by the aforementioned rule. The
control policy $d$ is a mapping from $(X_t,b_t)$ to $U_t$, i.e.,
$U_t = d(X_t,b_t)$. The output power of the system is determined by
$Y_t = X_t + U_t$. Our goal is to find the optimal value of $U_t$ at
every state $(X_t,b_t)$ such that the mean-variance combined metric
$\mathbb{E}[Y_t]-\beta\sigma^2[Y_t]$ can be maximized. The control
procedure of this problem is illustrated in Fig.~\ref{fig_wind_a}.

\begin{figure}[htbp]
\centering
\includegraphics[width=0.9\columnwidth]{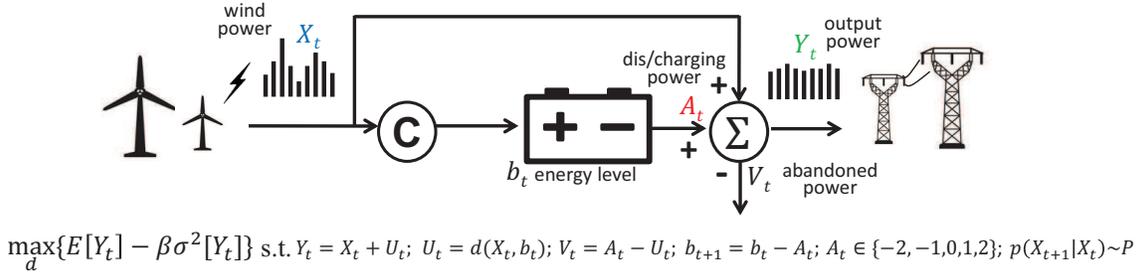}
\caption{The control of renewable energy with storage systems with
wind abandonment.}\label{fig_wind_a}
\end{figure}

We use Algorithm~\ref{algo1} to maximize the mean-variance combined
metric of this problem. We arbitrarily choose an initial policy from
the policy space. For different initial policies,
Algorithm~\ref{algo1} may converge to different local optima.
Fig.~\ref{fig_ite2} shows the convergence procedure when
$\beta=0.5$. It can be observed that the variance of power output
$\sigma^2[Y_t]$ is strictly reduced during each iteration, while the
average power output $\mathbb{E}[Y_t]$ has an increasing trend. The
combined metric $\mathbb{E}[Y_t]-\beta\sigma^2[Y_t]$ is continually
improved after several iterations until Algorithm~\ref{algo1}
converges. Algorithm~\ref{algo1} makes a balance between the wind
abandonment and the fluctuation reduction, which demonstrates the
effectiveness of Algorithm~\ref{algo1} for maximizing
$\mathbb{E}[Y_t]-\beta\sigma^2[Y_t]$. It is also observed that
Algorithm~\ref{algo1} converges after 5 to 7 iterations in most
cases, which shows the fast convergence speed of
Algorithm~\ref{algo1}.

\begin{figure}[htbp]
    \centering
    \includegraphics[width=0.55\columnwidth]{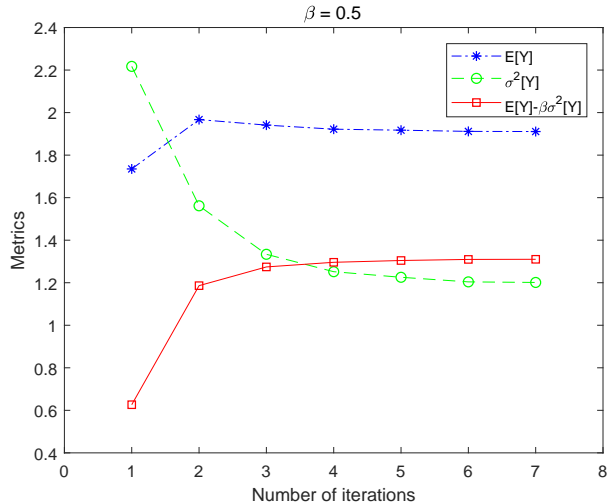}
    \caption{The convergence procedure of Algorithm~\ref{algo1} with wind abandonment.}\label{fig_ite2}
\end{figure}

Moreover, we investigate the convergence results of
Algorithm~\ref{algo1} under different initial policies and different
coefficients $\beta$. In each subgraph of Fig.~\ref{fig_pareto2}, we
randomly choose 5 different initial policies to implement
Algorithm~\ref{algo1}. Different curves indicate different
convergence trajectories under different initial policies. For the
case of $\beta=1$, we observe that Algorithm~\ref{algo1} converges
to two different local optima; For other cases, different initial
policies converge to the same local optimum. Every circle in
Fig.~\ref{fig_pareto2} represents a feasible solution. We prefer the
solutions with large $\mathbb{E}[Y_t]$ and small $\sigma^2[Y_t]$.
Some of the feasible solutions are the Pareto solutions, which
dominate other solutions in either $\mathbb{E}[Y_t]$ or
$\sigma^2[Y_t]$. All the Pareto solutions form the Pareto frontier,
which can be partly reflected in Fig.~\ref{fig_pareto2}.

\begin{figure}[htb]
    \centering
    \includegraphics[width=0.9\columnwidth]{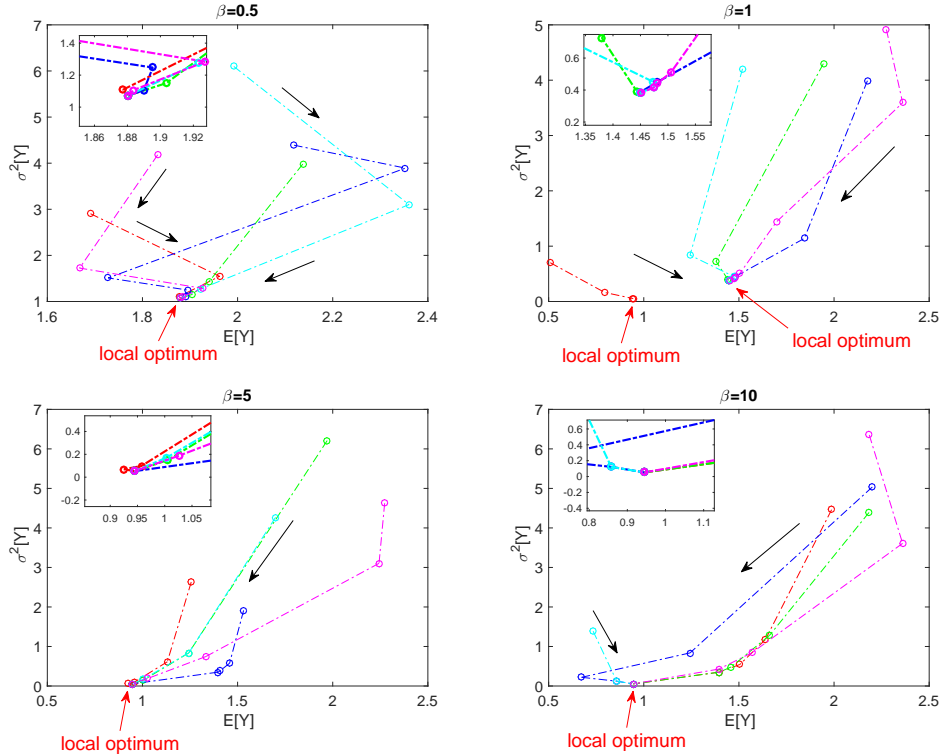}
    \caption{The convergence results of different initial policies and different $\beta$.}\label{fig_pareto2}
\end{figure}

Note that the solution space is too large for us to enumerate every
initial policies in Fig.~\ref{fig_pareto2}. We only choose some
representative solutions to outline the convergence procedures. The
top-left corner of each subgraph is the zooming-in area around the
convergence points. In future research, it is valuable to further
apply the exploration mechanisms such that the algorithm has
capability to jump out from local optima.

One intuitive way of exploration is to randomly sample the initial
policy such that Algorithm~\ref{algo1} can start at different
initial points and converge to different local optima. We can record
the optimal policy best so far until the computation budget is
depleted. In order to make the initial points as diverse as
possible, we may define proper metrics to measure the diversity of a
policy set. One example of heuristic ways to define the diversity
metric of a policy set $\mathcal D_0 \subseteq \mathcal D$ is as
follows.
\begin{equation}\label{eq_psi}
\Psi (\mathcal D_0) := \sum_{s \in \mathcal S}\Big | \bigcup_{d \in
\mathcal D_0} d(s) \Big|,
\end{equation}
where $|\cdot|$ is the number of elements in the set and
(\ref{eq_psi}) means the sum of the numbers of different actions at
all states. Thus, given the size of initial policy set $\mathcal
D_0$, if $\mathcal D_0$ is more diverse, we expect that
Algorithm~\ref{algo1} may converge to different local optima and the
best one is more likely the global optimum.

Another intuitive way is to introduce exploration schemes during the
step of policy generation in Algorithm~\ref{algo1}. That is, we
modify (\ref{eq_PI}) such that the generated new policy can have
more diversity. One feasible way is to use the $\epsilon$-greedy
scheme which is widely adopted in reinforcement learning
\citep{Sutton18}. At each state $s$, we adopt the action determined
by (\ref{eq_PI}) with probability $1-\epsilon$ and randomly select
actions from $\mathcal A$ with probability $\epsilon$, where
$\epsilon$ is a small positive number such as $\epsilon=0.05$. We
can record and output the best so far $J_{\mu,\sigma}$ until the
computation budget is depleted. Another more effective way is to use
the UCB (upper confident bound) method to balance the exploitation
and exploration during the search procedure
\citep{Agrawal95,Auer02}. We can define a counter $n(s,a)$ which
records the number of pair $(s,a)$ evaluated during algorithm
execution. We replace (\ref{eq_PI}) with the following scheme
\begin{equation}
d^{(l+1)}(i) := \argmax\limits_{a \in \mathcal A}\Bigg\{ r(i,a) -
\beta[r(i,a)-J_{\mu}]^2 + \sum_{j \in \mathcal S}p^a(i,j)
g_{\mu,\sigma}(j) + \gamma \sqrt{\frac{2 \ln
(\sum_a{n(s,a)})}{n(s,a)}} \Bigg\}, \nonumber
\end{equation}
where the square root part reflects the exploration benefit and
$\gamma$ is a coefficient balancing the exploitation and
exploration. Such scheme is effective especially considering the
estimation errors of $g_{\mu,\sigma}(j)$'s. We can let the
coefficient $\gamma$ vanishing when the estimation of
$g_{\mu,\sigma}(j)$ becomes more accurate. Similar schemes are
widely used in sample-based optimization algorithms, such as
Monte-Carlo tree search in AlphaGo \citep{Silver16} and adaptive
sampling in MDPs \citep{Chang07}. We can also resort to other
exploration techniques such as local search and global search
techniques widely used in evolutionary algorithms.

\section{Discussion and Conclusion} \label{section_conclusion}
The mean-variance combined metrics reflect both the average
performance and the risk-related performance. Since the variance
function is not additive, this mean-variance combined optimization
problem does not fit the standard model of MDPs. The classical
method of dynamic programming is not applicable. We study this
problem from a new perspective called the theory of
sensitivity-based optimization. The performance difference formula
is established to directly quantify the difference of the
mean-variance combined metrics under any two policies. The necessary
condition of the optimal policy is derived. The optimality of
deterministic policies is proved. We also develop a policy iteration
type algorithm to optimize the mean-variance combined metric. The
convergence of the algorithm is studied. Similar to the traditional
policy gradient approach, our approach also converges to a local
optimum in the mixed and randomized policy space, but with a more
efficient way as it has a form of policy iteration. The global
convergence of the algorithm is also discussed with the special
condition in Remarks~2\&4 or adopting some exploration and sampling
techniques. Experiment examples of the fluctuation reduction of wind
power with battery energy storage are conducted to demonstrate the
effectiveness of our approach.

One of the future research topics is to implement our approach in a
data-driven mode. This is a promising direction to study the
risk-sensitive reinforcement learning algorithm, since the variance
metric can reflect the risk-related factors. The online algorithm
implementation and the integration with neural networks deserve
further investigations, which is important to handle the issues of
model-absence and the curse of dimensionality. Another future
research topic is to extend our approach to optimize higher moment
metrics or even distribution optimization, since the variance is
only a second moment metric. For example, the third and the fourth
moment of rewards are also interesting metrics in statistics, which
reflect the skewness and kurtosis of reward distributions. One
recent work makes a good initiate on the
mean-variance-skewness-kurtosis analysis for the classic newsvendor
problem in a static optimization regime \citep{Zhang20}, while a
more general study in a wider field and dynamic optimization regime
will be a significant and challenging research topic.


\end{document}